\documentclass[11pt]{amsart}

\usepackage{amsaddr}
\usepackage{caption}
\usepackage{enumerate}
\usepackage{float}
\usepackage{graphicx}
\usepackage{geometry} 
\usepackage{setspace}
\usepackage{subcaption}
\usepackage{tabu}

\newcommand{\vx}{\vct{x}}

\newcommand{\vu}{\vct{u}}
\newcommand{\vvu}{\vvct{u}}

\newcommand{\vvv}{\vvct{v}}

\newcommand{\vw}{\vct{w}}

\newcommand{\valpha}{\boldsymbol\alpha}

\newcommand{\vkappa}{\boldsymbol\kappa}

\newcommand{\tsigma}{\boldsymbol{\sigma}}

\newcommand{\ttau}{\boldsymbol{\tau}}

\newcommand{\im}{\hat{\imath}}

\newcommand{\vct}[1]{\mathbf{#1}}
\newcommand{\vvct}[1]{\underline{#1}}
\newcommand{\ten}[1]{\mathbf{#1}}

\newcommand{\ccref}[1]{(\ref{#1})}

\newtheorem{thm}{Theorem}[section]
\newtheorem{lem}[thm]{Lemma}

\title[Efficient quadrature rules for mass-lumped tetrahedra]{Efficient quadrature rules for computing the stiffness matrices of mass-lumped tetrahedral elements for linear wave problems*}
\author{S. Geevers$^1$, W.A. Mulder$^{2,3}$ \and J.J.W. van der Vegt$^1$}
\address{1. Department of Applied Mathematics, University of Twente, Enschede, the Netherlands}
\address{2. Shell Global Solutions International BV}
\address{3. Delft University of Technology}
\thanks{*This work was funded by the Shell Global Solutions International B.V. under contract no.~PT45999.}

\begin{document}

\maketitle

\begin{abstract}
We present new and efficient quadrature rules for computing the stiffness matrices of mass-lumped tetrahedral elements for wave propagation modelling. These quadrature rules allow for a more efficient implementation of the mass-lumped finite element method and can handle materials that are heterogeneous within the element without loss of the convergence rate. The quadrature rules are designed for the specific function spaces of recently developed mass-lumped tetrahedra, which consist of standard polynomial function spaces enriched with higher-degree bubble functions. For the degree-2 mass-lumped tetrahedron, the most efficient quadrature rule seems to be an existing 14-point quadrature rule, but for tetrahedra of degrees 3 and 4, we construct new quadrature rules that require less integration points than those currently available in the literature. Several numerical examples confirm that this approach is more efficient than computing the stiffness matrix exactly and that an optimal order of convergence is maintained, even when material properties vary within the element.
\end{abstract}

\section{Introduction}
Mass-lumped tetrahedral element methods are efficient methods for solving linear wave equations, such as the acoustic wave equation, the elastic wave equations, or Maxwell's equations, on complex 3D domains with sharp material interfaces \cite{zhebel14}. They offer the same convergence rate and geometric flexibility as standard continuous tetrahedral element methods, but also allow for explicit time-stepping due to a diagonal mass matrix. 

To obtain mass-lumped elements, Lagrangian basis functions are combined with an inexact quadrature rule for computing the mass matrix. A lumped matrix is obtained when the quadrature points coincide with the basis function nodes. For quadrilateral and hexahedral elements, mass-lumping is achieved using tensor-product basis functions and Gauss--Lobatto points, resulting in the well-known spectral element method \cite{patera84, seriani94, komatitsch98}. For linear triangular and tetrahedral elements, mass-lumping is achieved with standard Lagrangian basis functions and a Newton--Cotes integration rule. For higher-degree triangular and tetrahedral elements, however, the element space needs to be enriched with higher-degree bubble functions in order to maintain stability and an optimal order of convergence \cite{cohen02}. So far, mass-lumped triangular elements of degree 2 and 3 \cite{cohen95,cohen01}, 4 \cite{mulder96}, 5 \cite{chin99}, 6 \cite{mulder13}, and 7-9 \cite{liu17, cui17} have been found. The first higher-degree tetrahedral elements were presented in \cite{mulder96} for degree 2 and in \cite{chin99} for degree 3. Recently, we presented new mass-lumped tetrahedral elements of degrees 2 to 4 in \cite{geevers18b}. The new degree-2 and degree-3 elements require significantly less nodes than the earlier versions, while mass-lumped tetrahedra of degree 4 had not been found before. Because of the reduced number of nodes, these new mass-lumped elements are also much more efficient than the earlier versions \cite{geevers18b} and are therefore more suitable for large-scale 3D simulations.

A question that remains is how to efficiently compute the stiffness matrix for these elements. When the material parameters are piecewise constant, the stiffness matrix can be evaluated exactly \cite{mulder16}. Alternatively, we can use a quadrature rule to approximate the stiffness matrix. This latter approach can significantly reduce the number of computations as we will demonstrate in Section \ref{sec:tests}. Moreover, it also allows us to handle material parameters that vary within the element without loss of convergence rate as we will prove in Section \ref{sec:accuracy}. 

Finding an efficient quadrature rule for the stiffness matrix for mass-lumped tetrahedra is not straightforward. For hexahedral elements, the stiffness matrix can be approximated with the same Gauss--Lobatto quadrature rule that is used for the mass matrix, but for mass-lumped tetrahedra, using the quadrature rule of the mass matrix to also evaluate the stiffness matrix turns out to be inefficient or inaccurate. In this paper, we therefore present new and efficient quadrature rules for computing the stiffness matrices for mass-lumped tetrahedra.

To obtain these quadrature rules, we show that the quadrature rule only needs to be exact for functions in the space $\mathcal{P}_{p-1}\otimes D\tilde U$, where $p\geq 2$ denotes the degree of the element, $\mathcal{P}_{p-1}$ denotes the space of polynomials up to degree $p-1$, and $D\tilde U$ denotes the space of partial derivatives of functions in the element space. Since the mass-lumped tetrahedra contain higher-degree bubble functions, so does the space $\mathcal{P}_{p-1}\otimes D\tilde U$ that needs to be integrated exactly. Most quadrature rules in literature, however, are designed to be exact for spaces of the form $\mathcal{P}_k$. Such quadrature rules for tetrahedral domains can be found in, for example, \cite{stroud71,grundmann78,keast86,cools03,zhang09,witherden15} and the references therein. We could choose $k$ equal to the highest polynomial degree that appears in $\mathcal{P}_{p-1}\otimes D\tilde U$, but the resulting number of quadrature points may then be suboptimal. Instead, we try to find quadrature rules that are exact for $\mathcal{P}_{p-1}\otimes D\tilde U$ with a minimal number of quadrature points.

For the degree-2 tetrahedral element, the most efficient quadrature rule still seems to be the 14-point rule of \cite{grundmann78} that is accurate for polynomials up to degree 5. For the degree-3 element and the three degree-4 elements of 60, 61, and 65 nodes, however, we present new quadrature rules that require 21, 51, 60, and 60 points, respectively, while using a quadrature rule from the current literature would require 24 \cite{keast86}, 59 \cite{zhang09}, 79 \cite{zhang09}, and 79 \cite{zhang09} points.

This paper is organised as follows. In Section \ref{sec:method}, we introduce the mass-lumped finite element method, and in Section \ref{sec:quadRules}, we present our new quadrature rules for evaluating the stiffness matrix. We prove in Section \ref{sec:accuracy} that the conditions used to obtain our quadrature rules result in optimal convergence rates. In Section \ref{sec:disp}, we analyze and compare the dispersion properties and resulting time step size for our new quadrature rules and several other rules available in the literature. In Section \ref{sec:tests}, we show numerical examples demonstrating that using our numerical quadrature rules for evaluating the stiffness matrix is more efficient than evaluating the integrals exactly and that the convergence rate is not lost when material parameters vary within the elements. Finally, we summarise our main conclusions in Section \ref{sec:conclusion}.

\section{The mass-lumped finite element method}
\label{sec:method}
To present and analyze the mass-lumped finite element method, we consider the scalar wave equation given by
\begin{subequations}
\begin{align}
\rho\partial_t^2u &= \nabla\cdot c\nabla u + f &&\text{in }\Omega\times(0,T),  \label{eq:modela}\\
u&= 0 &&\text{on }\partial\Omega\times(0,T), \\
u|_{t=0} &= u_0 &&\text{in }\Omega, \\
\partial_tu|_{t=0} &=v_0 &&\text{in }\Omega,
\end{align}%
\label{eq:model}%
\end{subequations}
where $\Omega\subset\mathbb{R}^3$ is the spatial domain, $(0,T)$ is the time domain, $u:\Omega\times(0,T)\rightarrow\mathbb{R}$ is the scalar field that needs to be solved, $\nabla$ is the gradient operator, $f:\Omega\times(0,T)\rightarrow\mathbb{R}$ is the source term, $u_0,v_0:\Omega\rightarrow\mathbb{R}$ are the initial values, and $\rho,c:\Omega\rightarrow\mathbb{R}^+$ are positive spatial parameters. The spatial domain $\Omega$ is assumed to be a bounded open domain with Lipschitz boundary $\partial\Omega$, and the parameters $\rho$ and $c$ are assumed to be bounded by $\rho_0\leq \rho\leq \rho_1$ and $c_0\leq c\leq c_1$, with $\rho_0,\rho_1,c_0,c_1$ strictly positive constants.

To solve the scalar wave equation with a finite element method, we consider the weak formulation. Let $L^2(\Omega)$ denote the standard Lesbesque space of square-integrable functions on $\Omega$, $H^1_0(\Omega)$ the standard Sobolev space of functions in $L^2(\Omega)$ that vanish on $\partial\Omega$ and have square-integrable weak derivatives, and $L^2(0,T;U)$, with $U$ a Banach space, the Bochner space of functions $f:(0,T)\rightarrow U$ such that $\|f\|_U$ is square integrable on $(0,T)$. Assume $u_0\in H^1_0(\Omega)$, $v_0\in L^2(\Omega)$, and $f\in L^2(0,T;L^2(\Omega))$. The weak formulation of \ccref{eq:model} can then be written as finding $u\in L^2\big(0,T;H^1_0(\Omega)\big)$, with $\partial_tu\in L^2\big(0,T;L^2(\Omega)\big)$ and $\partial_t(\rho\partial_tu)\in L^2\big(0,T;H^{-1}(\Omega)\big)$, such that $u|_{t=0}=u_0$, $\partial_tu|_{t=0}=v_0$, and
\begin{align}
\langle \partial_t(\rho\partial_tu), w\rangle + (c\nabla u,\nabla w) &= (f,w) &&\text{for all }w\in H_0^1(\Omega),\text{ a.e. }t\in(0,T),
\label{eq:WF}
\end{align}
where $(\cdot,\cdot)$ denotes the standard $L^2(\Omega)$ inner-product and $\langle\cdot,\cdot\rangle$ denotes the pairing between $H^{-1}(\Omega)$ and $H^1_0(\Omega)$.

This weak form of the wave equation can be solved with the mass-lumped finite element method, which consists of the following components:
\begin{itemize}
  \item a tetrahedral mesh $\mathcal{T}_h$, where $h$ denotes the radius of the smallest sphere that can contain each element,
  \item a reference tetrahedron $\tilde e$ with reference space $\tilde U=\mathcal{P}_p\oplus\tilde U^+:=\{u\;|\; u=w+u^+ \text{ for some }w\in\mathcal{P}_p, u^+\in\tilde{U}^+\}$, where $\mathcal{P}_p$ denotes the space of polynomials of degree $p$ or less and $\tilde{U}^+$ a space of higher-degree face and interior bubble functions,
  \item a set of reference nodes $\tilde{\mathcal Q}$ that can be used for both interpolation and quadrature on $\tilde e$,
  \item a set of quadrature weights $\{\tilde \omega_{\tilde\vx}\}_{\tilde\vx\in\tilde{\mathcal{Q}}}$.
\end{itemize}

Using these components, a finite element space can be constructed of the form 
\begin{align*}
U_h=H_0^1(\Omega)\cap U(\mathcal{T}_h,\tilde U),
\end{align*}
where
\begin{align*}
U(\mathcal{T}_h,\tilde U) &:= \{u\in H^1(\Omega) \;|\; u\circ\phi_e \in \tilde U \text{ for all }e\in\mathcal{T}_h\},
\end{align*}
with $\phi_e:\tilde e\rightarrow e$ the reference-to-physical element mapping. The interpolation points are given by $\mathcal{Q}_h=\mathcal{Q}(\mathcal{T}_h, \tilde{\mathcal{Q}})$, where
\begin{align*}
\mathcal{Q}(\mathcal{T}_h,\tilde{\mathcal{Q}}) &:= \bigcup_{e\in\mathcal{T}_h} \bigcup_{\tilde\vx\in\tilde{\mathcal{Q}}} \phi_e(\tilde\vx),
\end{align*} 
and the $L^2(\Omega)$ inner-product is approximated by
\begin{align*}
(u,w) = \sum_{e\in\mathcal{T}_h} \frac{|e|}{|\tilde e|} \int_{\tilde e} \tilde{u}_e\tilde{w}_e \;d\tilde x 
&\approx  \sum_{e\in\mathcal{T}_h} \sum_{\tilde\vx\in\tilde{\mathcal{Q}}} \frac{|e|}{|\tilde e|} \omega_{\tilde\vx} \tilde{u}_e(\tilde\vx)\tilde{w}_e(\tilde\vx) =:(u,w)_{\mathcal{Q}_{h}},
\end{align*}
with $|e|$ and $|\tilde e|$ the volume of $e$ and $\tilde e$, respectively, and $\tilde{u}_e:=u\circ\phi_e$, $\tilde{w}_e:=w\circ\phi_e$.

Assume $u_0,v_0,\rho,c\in\mathcal{C}^0(\overline\Omega)$ and $f:\mathcal{C}^0(\overline\Omega\times[0,T])$ are all continuous. The finite element method can then be formulated as finding $u_h:[0,T]\rightarrow U_h$ such that $u_h|_{t=0}=I_hu_0$, $\partial_tu_h|_{t=0}=I_hv_0$, and 
\begin{align}
(\rho\partial_t^2u_h,w)_{\mathcal{Q}_h} + (c\nabla u_h,\nabla w) &= (f,w)_{\mathcal{Q}_h} &&\text{for all }w\in U_h, t\in[0,T],
\label{eq:FEM}
\end{align}
where $I_h$ is the interpolation operator that interpolates a continuous function at the points $\mathcal{Q}_h$ by a function in $U(\mathcal{T}_h,\tilde U)$.

Now let $\{\vx_i\}_{i=1}^N$ be the set of all interpolation points $\mathcal{Q}_h$ that do not lie on the boundary $\partial\Omega$, and define nodal basis functions $\{w_i\}_{i=1}^N$ such that $w_i(\vx_j)=\delta_{ij}$ for all $i,j=1,\dots,N$, with $\delta$ the Kronecker delta. Also define, for any continuous function $u\in \mathcal{C}(\overline\Omega)$, the interpolation vector $\vvu\in\mathbb{R}^N$ such that $\vvu_i:=u(\vx_i)$ for all $i=1,\dots,N$. The finite element method can then be formulated as finding $\vvu_h:[0,T]\rightarrow\mathbb{R}^N$ such that $\vvu_h|_{t=0}=\underline{u_0}$, $\partial_t\vvu_h|_{t=0}=\underline{v_0}$, and
\begin{align}
\partial_t^2\vvu_h + M^{-1}A\vvu_h &= \vvct{\rho^{-1}f} &&\text{for all }t\in[0,T].
\label{eq:ODE}
\end{align}
Here, $M\in\mathbb{R}^{N\times N}$, with $M_{ij}:=(\rho w_i,w_j)_{\mathcal{Q}_h}$, is the mass matrix, and $A\in\mathbb{R}^{N\times N}$, with $A_{ij}:=(c\nabla w_i,\nabla w_j)$, is the stiffness matrix.

Since the interpolation points and quadrature points coincide, the mass matrix is diagonal with entries $M_{ii} = (\rho w_i,1)_{\mathcal{Q}_h}$. Therefore, we can efficiently solve the system of ODE's in \ccref{eq:ODE} using an explicit time-stepping scheme. Standard conforming finite element methods do not result in a (block)-diagonal mass matrix and are therefore less suitable for solving wave equations on large three-dimensional meshes.

To remain accurate and stable, the mass-lumped finite element method needs to satisfy the following conditions \cite{geevers18b}:
\begin{enumerate}
  \item[C1] (Unisolvent). The space $\tilde U$ is unisolvent on the nodes $\tilde{\mathcal{Q}}$.
  \item[C2] (Symmetry). The space $\tilde U$ and the set $\tilde{\mathcal{Q}}$ are invariant to affine mappings that map $\tilde e$ onto itself.
  \item[C3] (Face-conforming). If $\tilde u\in\tilde U$ is zero at all nodes in $\tilde{\mathcal{Q}}\cap\tilde f$, with $\tilde f$ a reference face, then $\tilde u$ is zero on $\tilde f$.
  \item[C4] (Positivity). The weights $\{\tilde \omega_{\tilde\vx}\}_{\tilde\vx\in\tilde{\mathcal{Q}}}$ are all strictly positive.
  \item[C5] (Accuracy). The quadrature rule is exact for functions in $\mathcal{P}_{p-2}\otimes\tilde U$ when $p\geq 2$.
\end{enumerate}
The first three conditions are necessary to guarantee that the global basis functions are well-defined and continuous. The last two conditions are necessary for stability and for maintaining an optimal order of convergence.

When $p\geq 2$, these conditions can not all be met for standard polynomial spaces $\tilde U=\mathcal{P}_p$. Therefore, the element space needs to be enriched with higher-degree bubble functions. We will focus on the mass-lumped tetrahedral elements recently presented in \cite{geevers18b}. An overview of these elements is given in Table \ref{tab:MLtet}. There, $n$ denotes the dimension of $\tilde U$, which is equal to the number of nodes per element, $B_f:=\mathrm{span}\{x_1x_2x_3,x_1x_2x_4,x_1x_3x_4,x_2x_3x_4\}$ denotes the span of the four face bubble functions, and $B_e:=\mathrm{span}\{x_1x_2x_3x_4\}$ denotes the span of the element bubble function, with $x_1$, $x_2$, $x_3$, $x_4$ the four barycentric coordinates. We also used the notation $UV:=U\otimes V:=\{w\;|\;w=uv \text{ for some }u\in U, v\in V\}$ for any two function spaces $U,V$.

\begin{table}[h]
\caption{Overview of mass-lumped tetrahedra.}
\label{tab:MLtet}
\begin{center}
\begin{tabular}{c| c| l}
$p$			& $n$ 	& $\tilde U$		\\ \hline\hline
2			& $15$ 	& $\mathcal{P}_2\oplus B_f\oplus B_e$ \\ \hline
3			& $32$ 	& $\mathcal{P}_3\oplus B_f\mathcal{P}_1\oplus B_e\mathcal{P}_1$ \\ \hline
4			& $60$	& $\mathcal{P}_4\oplus B_f\mathcal{P}_2 \oplus B_e(\mathcal{P}_2+B_f)$ \\ 
 			& $61$ 	& $\mathcal{P}_4\oplus B_f\mathcal{P}_2 \oplus B_e(\mathcal{P}_2+B_f+B_e)$ \\ 
 			& $65$ 	& $\mathcal{P}_4\oplus B_f(\mathcal{P}_2+B_f) \oplus B_e(\mathcal{P}_2+B_f+B_e)$ 
\end{tabular}
\end{center}
\end{table}

To apply these elements more efficiently, we also approximate the $L^2$ inner-product for the stiffness matrix, $(c\nabla u,\nabla v)$, with a quadrature rule. This also allows us to handle material parameters $c$ that vary within the element. It turns out that it is more efficient and sometimes even necessary to compute the stiffness matrix with a different quadrature rule than for the mass matrix. We will denote the quadrature points and weights for the stiffness matrix by $\tilde{\mathcal{Q}}'$ and $\{\tilde \omega'_{\tilde\vx}\}_{\tilde\vx\in\tilde{\mathcal{Q}}'}$, respectively, and denote the corresponding approximated $L^2(\Omega)$-product by $(\cdot,\cdot)_{\mathcal{Q}'_h}$.

The resulting finite element method remains stable and accurate if the following conditions are also satisfied:
\begin{enumerate}
  \item[C6] (Positivity). The weights $\{\tilde \omega_{\tilde\vx}'\}_{\tilde\vx\in\tilde{\mathcal{Q}}'}$ are all strictly positive.
  \item[C7] (Spurious-free). There is no function $\tilde u\in\tilde U$ with zero gradient $\tilde\nabla\tilde u=\vct{0}$ on all quadrature points $\tilde{\mathcal{Q}}'$ except the constant function. In case of linear elasticity, there is no function $\tilde{\vct{u}}\in \tilde{U}^3$ with zero strain $\tilde\nabla\tilde{\vct{u}} + \tilde\nabla\tilde{\vct{u}}^t = \ten{0}$ on all quadrature points $\tilde{\mathcal{Q}}'$ except the six rigid motions.
  \item[C8] (Accuracy). If $p\geq 2$, the quadrature rule for the stiffness matrix is exact for functions in $\mathcal{P}_{p-1}\otimes D\tilde U$, where $D \tilde U$ denotes the space of all partial derivatives of all functions in $\tilde U$. 
\end{enumerate}
A proof that these three conditions are sufficient to maintain an optimal order of convergence is given in Section \ref{sec:accuracy}.

We constructed quadrature rules that satisfy these conditions for the specific function spaces of the higher-degree mass-lumped tetrahedra presented in Table \ref{tab:MLtet}. For the degree-2 element, the most efficient quadrature rule seems to be an existing 14-point rule that is fifth-order accurate, but for the higher-degree elements, we obtained new quadrature rules that require less points than those currently available in the literature. An overview of these rules is given in the next section.

\section{Efficient quadrature rules for the stiffness matrix}
\label{sec:quadRules}

To present the quadrature rules for the stiffness matrix, let $\tilde e$ be the reference tetrahedron with vertices at $(0,0,0)$, $(1,0,0)$, $(0,1,0)$, and $(0,0,1)$. The barycentric coordinates of this element are given by the three Cartesian coordinates $x_1$, $x_2$, $x_3$, and the fourth coordinate $x_4:=1-x_1-x_2-x_3$. These coordinates are useful for describing $\mathcal{S}$, the space of affine mappings that map $\tilde e$ onto itself, since any function $s\in\mathcal{S}$ can be defined by a permutation of the barycentric coordinates. In particular, we can write $s(x_1,x_2,x_3)=(x_i,x_j,x_k)$ for some $i,j,k\in\{1,2,3,4\}$, $i\neq j\neq k\neq i$, for any $s\in\mathcal{S}$.

\begin{table}[h]
\caption{Types of quadrature points. The third column shows the number of equivalent points for each type.}
\label{tab:eqClass}
\begin{center}
\begin{tabular}{c ||c |r |p{0.5\textwidth}}
Type		& Points & $\#$ & Description\\ \hline
$[4]$ & $\{(1/4,1/4,1/4)\}$ & 1 & centre of tetrahedron \\
$[3,1]$ & $\{(c,c,c)\}$ & 4 & on line through vertex and centre \\
$[2,2]$ & $\{(d,d,1/2-d)\}$ & 6 & on line through edge-midpoint and centre \\
$[2,1,1]$ & $\{(f_1,f_1,f_2)\}$ & 12 & on plane through centre and two vertices \\
$[1,1,1,1]$ & $\{(g_1,g_2,g_3)\}$ & 24 & arbitrary position
\end{tabular}
\end{center}
\end{table}

\begin{table}[h]
\caption{Quadrature rule of 14 points ($K_4=0$, $K_{31}=2$, $K_{22}=1$, $K_{211}=0$, $K_{1111}=0$) \cite{grundmann78} for the stiffness matrix of the degree-2 15-node tetrahedron.}
\label{tab:ML2n15}
\begin{center}
{\tabulinesep=0.5mm
\begin{tabu}{c r l l}
Nodes					& $\#$ 	& $\omega'$ 	& node parameters  \\ \hline
$\{(c_1,c_1,c_1)\}$			& $4$	& $0.01224884051939366$		& $0.09273525031089123$ \\
$\{(c_2,c_2,c_2)\}$			& $4$	& $0.01878132095300264$		& $0.3108859192633006$ \\ 
$\{(d,d,\frac12-d)\}$			& $6$	& $0.007091003462846911$		& $0.04550370412564965$ \\ \hline
 \multicolumn{4}{c}{$V=\mathcal{P}_5=\{x_1,x_1^2x_2,x_1^3x_2^2, \beta_fx_1,\beta_fx_1x_2,\beta_ex_1 \}$} \\ \hline
\end{tabu}}
\end{center}
\end{table}

\begin{table}[h]
\caption{New quadrature rule of 21 points ($K_4=1$, $K_{31}=2$, $K_{22}=0$, $K_{211}=1$, $K_{1111}=0$) for the stiffness matrix of the degree-3 32-node tetrahedron.}
\label{tab:ML3n32}
\begin{center}
{\tabulinesep=0.5mm
\begin{tabu}{c r l l}
Nodes					& $\#$ 	& $\omega'$ 	& node parameters  \\ \hline
$\{(c_1,c_1,c_1)\}$			& $4$	& $0.008382813462606309$		& $0.08360982293995379$ \\
$\{(c_2,c_2,c_2)\}$			& $4$	& $0.01062803097330636$		& $0.3195556046935656$ \\ 
$\{(f_1,f_1,f_2)\}$			& $12$	& $0.005973459577178217$		& $\begin{bmatrix} 0.06366100187501753 \\ 0.3362519222398494 \end{bmatrix}$ \\ 
$\{(\frac14,\frac14,\frac14)\}$	& $1$	& $0.01894177399687740$		& - \\ \hline
\multicolumn{4}{c}{$V=\mathcal{P}_5\oplus B_f\mathcal{P}_3$}  \\
\multicolumn{4}{c}{$=\{x_1,x_1^2x_2,x_1^3x_2^2, \beta_fx_1,\beta_fx_1^2x_2,\beta_f^2, \beta_ex_1,\beta_ex_1x_2 \}$}  \\  \hline
\end{tabu}}
\end{center}
\end{table}

\begin{table}[h]
\caption{New quadrature rule of 51 points ($K_4=1$, $K_{31}=2$, $K_{22}=1$, $K_{211}=3$, $K_{1111}=0$) for the stiffness matrix of the degree-4 60-node tetrahedron.}
\label{tab:ML4n60}
\begin{center}
{\tabulinesep=0.5mm
\begin{tabu}{c r l l}
Nodes					& $\#$ 	& $\omega'$ 	& node parameters  \\ \hline
$\{(c_1,c_1,c_1)\}$			& $4$	& $0.001076330088382485$		& $0.04010756377220036$ \\
$\{(c_2,c_2,c_2)\}$			& $4$	& $0.006422430307819483$		& $0.1881144601918900$ \\ 
$\{(d,d,\frac12-d)\}$			& $6$	& $0.003859721113202450$		& $0.1124010568611476$ \\ 
$\{(f_{11},f_{11},f_{12})\}$		& $12$	& $0.003162722714222902$		& $\begin{bmatrix} 0.04781990270450464 \\  0.2053222493389064 \end{bmatrix}$ \\ 
$\{(f_{21},f_{21},f_{22})\}$		& $12$	& $0.004715130256124021$		& $\begin{bmatrix} 0.2347999378738287 \\  0.03405863749492695 \end{bmatrix}$ \\ 
$\{(f_{31},f_{31},f_{32})\}$		& $12$	& $0.001320748780834370$		& $\begin{bmatrix} 0.4614535776221135 \\ 0.06693547308143162 \end{bmatrix}$ \\ 
$\{(\frac14,\frac14,\frac14)\}$	& $1$	& $0.003130077388468573$		& - \\ \hline
\multicolumn{4}{c}{$V=\mathcal{P}_7\oplus B_f(\mathcal{P}_5\oplus B_f\mathcal{P}_3) \oplus B_e\mathcal{P}_5$}  \\
\multicolumn{4}{c}{$=\{x_1,x_1^2x_2,x_1^3x_2^2,x_1^4x_2^3, \beta_fx_1,\beta_fx_1^2x_2,\beta_fx_1^3x_2^2, \beta_f^2x_1,\beta_f^2x_1^2x_2,\beta_f^3, \dots$}  \\  
\multicolumn{4}{c}{$\dots,  \beta_ex_1,\beta_ex_1^2x_2,\beta_ex_1^3x_2^2,\beta_e\beta_fx_1,\beta_e\beta_fx_1x_2, \beta_e^2x_1 \}$}  \\  \hline
\end{tabu}}
\end{center}
\end{table}

\begin{table}[h]
\caption{New quadrature rule of 60 points ($K_4=0$, $K_{31}=3$, $K_{22}=2$, $K_{211}=3$, $K_{1111}=0$) for the stiffness matrix of the degree-4 61- and 65-node tetrahedron.}
\label{tab:ML4n65}
\begin{center}
{\tabulinesep=0.5mm
\begin{tabu}{c r l l}
Nodes					& $\#$ 	& $\omega'$ 	& node parameters  \\ \hline
$\{(c_1,c_1,c_1)\}$			& $4$	& $0.001137453809249273$		& $0.04091036488546224$ \\
$\{(c_2,c_2,c_2)\}$			& $4$	& $0.006907244220995018$		& $0.1942594527940223$ \\ 
$\{(c_3,c_3,c_3)\}$			& $4$	& $0.004458749819772567$		& $0.3166409312612929$ \\ 
$\{(d_1,d_1,\frac12-d_1)\}$		& $6$	& $0.001389883779363477$		& $0.02776256108257648$ \\ 
$\{(d_2,d_2,\frac12-d_2)\}$		& $6$	& $0.004236295194116969$		& $0.1022199785693040$ \\ 
$\{(f_{11},f_{11},f_{12})\}$		& $12$	& $0.001788418107829456$		& $\begin{bmatrix} 0.03511432271187172 \\ 0.2097218125202450  \end{bmatrix}$ \\ 
$\{(f_{21},f_{21},f_{22})\}$		& $12$	& $0.003642034272731381$		& $\begin{bmatrix} 0.1790174868402900 \\ 0.03980830656880513  \end{bmatrix}$ \\ 
$\{(f_{31},f_{31},f_{32})\}$		& $12$	& $0.001477531071582210$		& $\begin{bmatrix} 0.4192720711456938 \\ 0.008950317872961031 \end{bmatrix}$ \\ \hline
\multicolumn{4}{c}{$V=\mathcal{P}_8\oplus B_f^2\mathcal{P}_3 \oplus B_e(\mathcal{P}_5\oplus B_f\mathcal{P}_3)$}  \\
\multicolumn{4}{c}{$=\{x_1,x_1^2x_2,x_1^3x_2^2,x_1^4x_2^3,x_1^4x_2^4, \beta_fx_1,\beta_fx_1^2x_2,\beta_fx_1^3x_2^2, \beta_f^2x_1,\beta_f^2x_1^2x_2,\beta_f^3, \dots$}  \\  
\multicolumn{4}{c}{$\dots,  \beta_ex_1,\beta_ex_1^2x_2,\beta_ex_1^3x_2^2,\beta_e\beta_fx_1,\beta_e\beta_fx_1^2x_2,\beta_e\beta_f^2, \beta_e^2x_1, \beta_e^2x_1x_2 \}$}  \\  \hline
\end{tabu}}
\end{center}
\end{table}

Now, let $\{\vx\}$ denote point $\vx$ and all equivalent points $s(\vx)$, with $s\in\mathcal{S}$. The quadrature rule will consist of several equivalence classes $\{\vx\}$ with quadrature weights that are the same within each equivalence class. To give an example of an equivalence class, consider the point $(c_1,c_1,c_1)$. The barycentric coordinates of this point are given by $c_1,c_1,c_1,1-3c_1$, so the equivalence class $\{(c_1,c_1,c_1)\}$ consists of the four points $(c_1,c_1,c_1)$, $(1-3c_1,c_1,c_1)$, $(c_1,1-3c_1,c_1)$, and $(c_1,c_1,1-3c_1)$ when $c_1\neq\frac14$. An overview of the different types of points is given in Table \ref{tab:eqClass}. The configuration of a quadrature rule is given by the numbers $K_{4}$, $K_{31}$, $K_{22}$, $K_{211}$, and $K_{1111}$, which indicate that the quadrature rule has $K_4$ distinct points of type $[4]$, $K_{31}$ points of type $[3,1]$, $K_{22}$ points of type $[2,2]$, etc.

To find a set of points and weights that satisfy accuracy condition C8, we construct a linear basis that spans $V\supset \mathcal{P}_{p-1}\otimes D\tilde U$. We describe this linear basis and linear span using the notation $\{f_1,f_2,\dots,f_k\}$, which denotes the span of the functions $f_1,\dots,f_k$ and all equivalent functions $f_i\circ s$, with $i=1,\dots,k$ and $s\in\mathcal{S}$. To give an example, all equivalent versions of $x_1^2x_2^2x_3x_4$ are given by the six functions $x_1^2x_2^2x_3x_4$, $x_1^2x_2x_3^2x_4$, $x_1^2x_2x_3x_4^2$, $x_1x_2^2x_3^2x_4$, $x_1x_2^2x_3x_4^2$, and $x_1x_2x_3^2x_4^2$, so $\{x_1^2x_2^2x_3x_4\}$ denotes the span of these six functions.

After having constructed a basis $\{f_1,f_2,\dots,f_k\}$ for $V$, we search for a quadrature rule that has a configuration with $k$ parameters. These parameters consist of location parameters and quadrature weights. Because of the symmetry, a quadrature rule that is exact for a function $f$ is exact for all its equivalent functions. Therefore, to satisfy C8, we end up with a nonlinear system of $k$ equations:
\begin{align}
\int_{\tilde e} f_i(\tilde\vx) \;d\tilde x &= \sum_{\tilde\vx\in\tilde{\mathcal{Q}}'} \omega'_{\tilde\vx} f_i(\tilde\vx),  &&i=1,\dots,k.
\label{eq:quadRule}
\end{align}
We obtain solutions of this system using Newton's method for a large number of different initial values and check for each solution if it satisfies C6 and C7. When we cannot find an admissible solution for configurations with $k$ parameters, we increase $k$ and try again. We continue this process until we find a suitable quadrature rule.

The complete algorithm for finding a quadrature rule of $k$ parameters can be summarised as follows:
\begin{enumerate}
  \item[1.] Construct basis functions $f_1,\dots,f_k$, such that $\{f_1,f_2,\dots, f_k\}\supset \mathcal{P}_{p-1}\otimes D\tilde U$
  \item[2.] Choose a configuration for the quadrature rule such that 
  \begin{align*}
   \text{\# parameters} = K_{4} + 2K_{31} + 2K_{22}  + 3K_{211} + 4K_{1111} = k,
  \end{align*}
  \item[3.] Solve the system of equations given in \ccref{eq:quadRule} using Newton's method with a random initial guess. 
  \item[4.] If the method converges within a maximum number of iterations, check if the solution satisfies conditions C6 and C7.
  \item[5.] Repeat Steps 3 and 4 until a maximum number of trials is reached or an admissible solution is found.
\end{enumerate}

The quadrature rules that were obtained in this way are given in Tables \ref{tab:ML2n15}-\ref{tab:ML4n65}. There, $\#$ denotes the number of nodes in each equivalence class and $\beta_f:=x_1x_2x_3$ and $\beta_e=x_1x_2x_3x_4$ denote the face bubble function and interior bubble function, respectively. 

The quadrature rule with the least number of points we could find for the degree-2 15-node tetrahedron is the 14-point fifth-order accurate rule of \cite{grundmann78}. We also found an accurate quadrature rule of 10 points with positive weights, but the resulting method was not accurate since condition C7 was not satisfied: it had one non-constant mode with gradient equal to zero at all 10 points. We also considered the 15-point quadrature rule used for the mass matrix, which also satisfies C6-C8. However, this rule significantly increases the condition number of the element matrix and therefore results in a considerably smaller time step size as shown in Section \ref{sec:disp}.

The quadrature rules for the degree-3 and degree-4 elements are new and require less quadrature points than rules currently available in the literature, since most quadrature rules in the literature are constructed to be exact for a function space of the form $\mathcal{P}_k$ and not for the specific function spaces $\mathcal{P}_{p-1}\otimes D\tilde U$. To give an example, the highest polynomial degree of $D\tilde U$ of the degree-4 61- or 65-node tetrahedron is 7, so $\mathcal{P}_3\otimes D\tilde U$ contains a polynomial of degree 10. A quadrature rule that is order-10 accurate already requires 79 quadrature points \cite{zhang09}, while our quadrature rule for these elements only requires 60 points. Similarly, our quadrature rules for the degree-3 32-node tetrahedron and the degree-4 60-node tetrahedron require 21 and 51 points, respectively, while the quadrature rules currently available in the literature for these elements require 24 and 59 points \cite{zhang09}.

\section{Error estimates}
\label{sec:accuracy}

In this section, we prove that, when conditions C1-C8 are satisfied, the finite element method maintains an optimal order of convergence for a related elliptic problem. Convergence for the wave equation can then be derived in a way analogous to \cite[Chapter 4.6]{geevers18b}.

Throughout this section, we will let $p$ denote the degree of the finite element space, by which we mean the largest degree such that $\tilde U\supset\mathcal{P}_p$. We will also let $C$ denote a positive constant that may depend on the domain $\Omega$, the regularity of the mesh, the parameters $\rho$ and $c$, the reference space $\tilde U$, and the reference quadrature rules $(\tilde{\mathcal{Q}},\{\tilde\omega_{\tilde\vx}\}_{\tilde\vx\in\tilde{\mathcal{Q}}})$ and $(\tilde{\mathcal{Q}}',\{\tilde\omega_{\tilde\vx}\}_{\tilde\vx\in\tilde{\mathcal{Q}}'})$, but does not depend on the mesh resolution $h$ and the functions that appear in the inequalities.

\subsection{Preliminary results}
To obtain error bounds, we first define some norms and function spaces and list a few preliminary results. Let $H^k(\Omega)$ denote the Sobolev space of functions on $\Omega$ with order-$k$ square-integrable weak derivatives and equip the space with norm
\begin{align*}
\|u\|_k^2 &:= \sum_{|\alpha|\leq k} \|D^{\valpha}u\|_0^2, 
\end{align*}
where $\|\cdot\|_0$ denotes the $L^2$-norm, $D^{\valpha}:=\partial_1^{\alpha_1}\partial_2^{\alpha_2}\partial_3^{\alpha_3}$ the partial derivative, and $|\valpha|:=\alpha_1+\alpha_2+\alpha_3$ the order of the derivative. We also define the broken Sobolev spaces $H^k(\mathcal{T}_h):=\{u\in L^2(\Omega) \;|\; u|_e\in H^k(e) \text{ for all }e\in\mathcal{T}_h\}$, equipped with norm
\begin{align*}
\|u\|^2_{\mathcal{T}_h,k}:=\sum_{e\in\mathcal{T}_h} \|u|_e\|_k^2.
\end{align*} 
Throughout this section, we will use the fact that $H^2(\Omega) \supset\mathcal{C}^0(\overline\Omega)$ for any three-dimensional domain $\Omega$. 

We also define the semi-norms $|u|_{\mathcal{Q}_h}^2:=(u,u)_{\mathcal{Q}_h}$ and $|\sigma|_{\mathcal{Q}_h'}^2:=(\sigma,\sigma)_{\mathcal{Q}_h'}$ for piecewise continuous functions $u$ and $\sigma$, and define $\Pi_{h,q}$ to be the $L^2$-projection operators projecting onto the discontinuous piecewise-polynomial spaces $V(\mathcal{T}_h,\mathcal{P}_{q}):=\{u\in L^2(\Omega) \;|\; u\circ\phi_e\in\mathcal{P}_q\text{ for all }e\in\mathcal{T}_h\}$. Several useful properties of these spaces and operators are listed below.

\begin{lem}
\label{lem:pre1}
Let $q\geq 0$. Then 
\begin{align*}
|u_h|_{\mathcal{Q}_h} &\leq C\|u_h\|_0 &&\text{for all }u_h\in V(\mathcal{T}_h,\mathcal{P}_{q}), \\
|\tsigma_h|_{\mathcal{Q}_h'} &\leq C\|\tsigma_h\|_0 &&\text{for all }\tsigma_h\in V(\mathcal{T}_h,\mathcal{P}_{q})^3.
\end{align*}
\end{lem}
\begin{proof}
These results follow immediately from the fact that the elements are affine equivalent with the reference element and that the reference space $\mathcal{P}_{q}$ is finite dimensional.
\end{proof}

\begin{lem}
\label{lem:pre2}
If conditions C1-C4 are satisfied, then $|\cdot|_{\mathcal{Q}_h}$ becomes a full norm $\|\cdot\|_{\mathcal{Q}_h}$ on $U(\mathcal{T}_h,\tilde U)$ and
\begin{align*}
 \|u_h\|_{\mathcal{Q}_h} \geq C\|u_h\|_0 &&\text{for all }u_h\in U_h.
\end{align*}
Furthermore, if conditions C1-C3, C6, and C7 are satisfied, then $|\cdot|_{\mathcal{Q}_h'}$ becomes a full norm $\|\cdot\|_{\mathcal{Q}_h'}$ on $V(\mathcal{T}_h,D\tilde U)$ and
\begin{align*}
\|\nabla u_h\|_{\mathcal{Q}_h'} \geq C\|\nabla u_h\|_{0} &&\text{for all }u_h\in U_h.
\end{align*}
\end{lem}

\begin{proof}
These inequalities follow immediately from the fact that the elements are affine equivalent with the reference element and that the reference element space $\tilde U$ is finite dimensional.
\end{proof}

\begin{lem}
\label{lem:pre3}
Let $u \in H^k(\mathcal{T}_h)$ and $\tsigma \in H^k(\mathcal{T}_h)^3$, with $k\geq 0$, and let $q\geq 0$. Then
\begin{align*}
\|u-\Pi_{h,q}u \|_{\mathcal{T}_h,m} &\leq Ch^{\min(q+1,k)-m}\|u\|_{\mathcal{T}_h,\min(q+1,k)}, &&m\leq\min(q+1,k), \\
\|\tsigma-\Pi_{h,q}\tsigma \|_{\mathcal{T}_h,m} &\leq Ch^{\min(q+1,k)-m}\|\tsigma\|_{\mathcal{T}_h,\min(q+1,k)}, &&m\leq\min(q+1,k).
\end{align*}
Furthermore, if $k\geq 2$, then
\begin{align*}
|u-\Pi_{h,q}u |_{\mathcal{Q}_h} &\leq Ch^{\min(q+1,k)}\|u\|_{\mathcal{T}_h,\min(q+1,k)},  \\ 
|\tsigma-\Pi_{h,q}\tsigma |_{\mathcal{Q}_h'} &\leq Ch^{\min(q+1,k)}\|\tsigma\|_{\mathcal{T}_h,\min(q+1,k)}.
\end{align*}
Finally, if $u\in H^1(\Omega)\cap H^k(\mathcal{T}_h)$, with $k\geq 2$, then
\begin{align*}
\|u-I_hu\|_{\mathcal{T}_h,m} &\leq Ch^{\min(p+1,k)-m}\|u\|_{\mathcal{T}_h,\min(p+1,k)}, &&m\leq\min(p+1,k)
\end{align*}
with $p\geq 2$ the degree of the finite element space.
\end{lem}

\begin{proof}
The first, second, and last inequality follow from \cite[Chapter 3.1]{ciarlet78}. For the fourth inequality, let $q^*\geq q$, be a polynomial degree and $\tilde{\mathcal{Q}}^*\supset\tilde{\mathcal{Q}}'$ a set of points such that $\mathcal{P}_{q^*}$ is unisolvent on $\tilde{\mathcal{Q}}^*$. Also, let $I_h^*$ be the interpolation operator that interpolates a function in $H^2(\mathcal{T}_h)$ at the nodes $\mathcal{Q}(\mathcal{T}_h,\tilde{\mathcal{Q}}^*)$ by a function in $V(\mathcal{T}_h,\mathcal{P}_{q^*})$. We can then obtain the fourth inequality as follows:
\begin{align*}
|\tsigma-\Pi_{h,q}\tsigma |_{\mathcal{Q}_h'} &= |I_h^*\tsigma-\Pi_{h,q}\tsigma |_{\mathcal{Q}_h'} \\
&\leq C \|I_h^*\tsigma-\Pi_{h,q}\tsigma \|_{0} \\
&\leq C (\|I_h^*\tsigma-\tsigma \|_{0} + \|\tsigma-\Pi_{h,q}\tsigma \|_{0}) \\
&\leq Ch^{\min(q+1,k)}\|\tsigma\|_{\mathcal{T}_h,\min(q+1,k)},
\end{align*}
where we used Lemma \ref{lem:pre1} in the the second line and the triangle inequality in the third line. The last line follows from \cite[Chapter 3.1]{ciarlet78}. 

The third inequality can be derived in a way analogous to the fourth inequality.
\end{proof}

\subsection{Estimates on the integration error}
Define the two integration errors $r_h(u,w):=(u,w)-(u,w)_{\mathcal{Q}_h}$ and $r'_h(\tsigma,\ttau):=(\tsigma,\ttau)-(\tsigma,\ttau)_{\mathcal{Q}_h'}$. In \cite{geevers18b} we derived the following bounds on $r_h$.
\begin{lem}
\label{lem:int1}
Let $p\geq 2$ be the degree of the finite element space, $u\in H^k(\Omega)$, with $k\geq 2$, and $w\in U_h$. If conditions C1-C5 are satisfied, then
\begin{align*}
|r_h(u,w)| &\leq Ch^{\min(p,k)}\|u\|_{\min(p,k)}\|w\|_1,  \\
|r_h(u,w)| &\leq Ch^{\min(p+1,k)}\|u\|_{\min(p+1,k)}\|w\|_{\mathcal{T}_h,2}.
\end{align*}
\end{lem}

We also derive bounds on the integration error for the stiffness matrix.
\begin{lem}
\label{lem:int2}
Let $p\geq 2$ be the degree of the finite element space, $\tsigma\in H^k(\mathcal{T}_h)^3$ with $k\geq 2$, and $\ttau\in V(\mathcal{T}_h,D\tilde U)^3$. If conditions C1-C3, C6, and C8 are satisfied, then
\begin{align}
|r_h'(\tsigma,\ttau)| &\leq Ch^{\min(p,k)}\|\tsigma\|_{\mathcal{T}_h,\min(p,k)}\|\ttau\|_{0}, \label{eq:int2a} \\
|r_h'(\tsigma,\ttau)| &\leq Ch^{\min(p+1,k)}\|\tsigma\|_{\mathcal{T}_h,\min(p+1,k)}\|\ttau\|_{\mathcal{T}_h,1}. \label{eq:int2b}
\end{align}
\end{lem}

\begin{proof}
Using C8, we can write
\begin{align*}
r_h'(\tsigma,\ttau) &= r_h'(\tsigma-\Pi_{h,p-1}\tsigma,\ttau). 
\end{align*}
Inequality \ccref{eq:int2a} then follows from the Cauchy--Schwarz inequality and Lemma \ref{lem:pre3}.

Using C8 and the fact that $\mathcal{P}_p\subset\mathcal{P}_{p-1}\otimes D\tilde U$ for $p\geq 2$, we can also write
\begin{align*}
r_h'(\tsigma,\ttau) &= r_h'\big((\tsigma-\Pi_{h,p}\tsigma)+(\Pi_{h,p}\tsigma-\Pi_{h,p-1}\tsigma) + \Pi_{h,p-1}\tsigma, (\ttau-\Pi_{h,0}\ttau) + \Pi_{h,0}\ttau \big) \\
&= r_h'(\tsigma-\Pi_{h,p}\tsigma,\ttau) + r_h'(\Pi_{h,p}\tsigma-\Pi_{h,p-1}\tsigma,\ttau-\Pi_{h,0}\ttau).
\end{align*}
Inequality \ccref{eq:int2b} then follows from the Cauchy--Schwarz inequality and Lemma \ref{lem:pre3}.
\end{proof}

\subsection{Error estimates for a related elliptic problem}
Let $v\in \mathcal{C}^0(\overline\Omega)$. The elliptic problem corresponding to \ccref{eq:WF} is finding $u\in H^1_0(\Omega)$ such that
\begin{align}
(c\nabla u, \nabla w) &= (v,w) &&\text{for all }w\in H^1_0(\Omega).
\label{eq:ell}
\end{align}
The corresponding mass-lumped finite element method is finding $u_h\in U_h$ such that
\begin{align}
(c\nabla u_h,\nabla w)_{\mathcal{Q}_h'} &= (v,w)_{\mathcal{Q}_h} &&\text{for all }w\in U_h.
\label{eq:ellh}
\end{align}
In the next two theorems we prove optimal convergence in the $H^1$-norm and $L^2$-norm.

\begin{thm}[Optimal convergence in the $H^1$-norm]
\label{thm:ell1}
Let $u$ be the solution of \ccref{eq:ell} and $u_h$ the solution of \ccref{eq:ellh}. Assume $c\in\mathcal{C}^p(\overline\Omega)$, $u\in H^{k_u}(\Omega)$, and $v\in H^{k_v}(\Omega)$, with $k_u,k_v\geq 2$. If conditions C1-C8 are satisfied, then
\begin{align}
\label{eq:ell1}
\|u-u_h\|_1 &\leq Ch^{\min(p,k_u-1,k_v)}(\|u\|_{\min(p+1,k_u)} + \|v\|_{\min(p,k_v)}),
\end{align}
with $p \geq 2$ the degree of the finite element method.
\end{thm}

\begin{proof}
Define $e_h:=I_hu-u_h$ and $\epsilon_h:=u-I_hu$. Using \ccref{eq:ell}, we can write 
\begin{align*}
(c\nabla I_hu,\nabla e_h)_{\mathcal{Q}_h'} =& -r_h'(c\nabla I_hu,\nabla e_h) - (c\nabla \epsilon_h,\nabla e_h) + (c \nabla u , \nabla e_h) \\
=& -r_h'(c\nabla I_hu,\nabla e_h)- (c\nabla \epsilon_h,\nabla e_h) + (v, e_h)
\end{align*}
and using \ccref{eq:ellh}, we can obtain
\begin{align*}
(c\nabla u_h,\nabla e_h)_{\mathcal{Q}_h'} = (v, e_h)_{\mathcal{Q}_h}.
\end{align*}
Subtracting these two equalities gives
\begin{align}
\label{eq:ell1a}
(c\nabla e_h,\nabla e_h)_{\mathcal{Q}_h'} = -r_h'(c\nabla I_hu,\nabla e_h)- (c\nabla \epsilon_h,\nabla e_h) + r_h(v, e_h).
\end{align}
From Lemma \ref{lem:pre2}, the positivity of $c$, and Poincar\'e's inequality, it follows that
\begin{align}
\label{eq:ell1b}
\|e_h\|_1^2 &\leq C(c\nabla e_h,\nabla e_h)_{\mathcal{Q}_h'} .
\end{align}
Using Lemma \ref{lem:int2}, Lemma \ref{lem:pre3}, and the regularity of $c$, we can obtain
\begin{align}
\label{eq:ell1c}
|r_h'(c\nabla I_hu,\nabla e_h)| &\leq Ch^{\min(p,k_u-1)}\|c\nabla I_hu\|_{\mathcal{T}_h,\min(p,k_u-1)}\|e_h\|_1 \nonumber \\
& \leq Ch^{\min(p,k_u-1)}\|u\|_{\min(p+1,k_u)}\|e_h\|_1.
\end{align}
Using the Cauchy--Schwarz inequality, the boundedness of $c$, and Lemma \ref{lem:pre3}, we can also obtain
\begin{align}
\label{eq:ell1d}
|(c\nabla \epsilon_h,\nabla e_h)| &\leq Ch^{\min(p,k_u-1)}\|u\|_{\min(p+1,k_u)}\|e_h\|_1.
\end{align}
Finally, using Lemma \ref{lem:int1}, we can obtain
\begin{align}
\label{eq:ell1e}
|r_h(v, e_h)| &\leq Ch^{\min(p,k_v)}\|v\|_{\min(p,k_v)}\|e_h\|_1.
\end{align}
Combining \ccref{eq:ell1a}-\ccref{eq:ell1e} gives
\begin{align*}
\|e_h\|_1 &\leq Ch^{\min(p,k_u-1,k_v)}(\|u\|_{\min(p+1,k_u)} + \|v\|_{\min(p,k_v)}).
\end{align*}
Since $u-u_h=e_h+\epsilon_h$, inequality \ccref{eq:ell1} then follows from the above and Lemma \ref{lem:pre3}.
\end{proof}

To prove optimal convergence in the $L^2$-norm, we make the following regularity assumption: for any $v\in L^2(\Omega)$, the solution of \ccref{eq:ell} is in $H^2(\Omega)$ and satisfies
\begin{align}
\label{eq:reg}
\|u\|_2 \leq C\|v\|_0.
\end{align}
This is certainly true when $\partial\Omega$ is $\mathcal{C}^2$ and $c\in\mathcal{C}^1(\overline\Omega)$.

\begin{thm}[Optimal convergence in the $L^2$-norm]
\label{thm:ell2}
Let $u$ be the solution of \ccref{eq:ell} and $u_h$ the solution of \ccref{eq:ellh}. Assume $c\in\mathcal{C}^{p+1}(\overline\Omega)$, $u\in H^{k_u}(\Omega)$, and $v\in H^{k_v}(\Omega)$, with $k_u,k_v\geq 2$, and assume the regularity condition \ccref{eq:reg} holds. If conditions C1-C8 are satisfied, then
\begin{align}
\label{eq:ell2}
\|u-u_h\|_{0} &\leq Ch^{\min(p+1,k_u,k_v)}(\|u\|_{\min(p+1,k_u)} + \|v\|_{\min(p+1,k_v)}),
\end{align}
with $p \geq 2$ the degree of the finite element method.
\end{thm}

\begin{proof}
Define $z_h\in H^1_0(\Omega)$ to be the solution of the elliptic problem
\begin{align*}
(c\nabla z_h,\nabla w) &= (u-u_h,w) &&\text{for all }w\in H_0^1(\Omega).
\end{align*}
From the regularity assumption it follows that $z_h\in H^2(\Omega)$ and 
\begin{align*}
\|z_h\|_2\leq C\|u-u_h\|_0.
\end{align*}
From the definition of $z_h$, it also follows that
\begin{align}
\label{eq:ell2a}
\|u-u_h\|_0^2 &= (c\nabla [u-u_h],\nabla z_h) \nonumber \\
&= (c\nabla [u-u_h],\nabla[z_h-I_hz_h]) + (c\nabla [u-u_h],\nabla I_hz_h).
\end{align}
We can bound the term $(c\nabla [u-u_h],\nabla[z_h-I_hz_h])$ as follows: 
\begin{align}
\label{eq:ell2b}
|(c\nabla [u-u_h],& \nabla[z_h-I_hz_h])| \leq C\|u-u_h\|_1\|z_h-I_hz_h\|_1 \nonumber \\
& \leq Ch^{\min(p,k_u-1,k_v)+1}(\|u\|_{\min(p+1,k_u)} + \|v\|_{\min(p,k_v)})\|z_h\|_2 \nonumber \\
& \leq Ch^{\min(p+1,k_u,k_v+1)}(\|u\|_{\min(p+1,k_u)} + \|v\|_{\min(p,k_v)})\|u-u_h\|_0,
\end{align}
where we used the Cauchy--Schwarz inequality and the boundedness of $c$ in the first line, Theorem \ref{thm:ell1} and Lemma \ref{lem:pre3} in the second line, and the regularity assumption in the last line. It then remains to find a bound for $(c\nabla [u-u_h], \nabla I_hz_h)$.

To do this, use \ccref{eq:ell} to write
\begin{align*}
(c\nabla u,\nabla I_hz_h) &= (v,I_hz_h)
\end{align*}
and use \ccref{eq:ellh} to write
\begin{align*}
(c\nabla u_h,\nabla I_hz_h) &= r_h'(c\nabla u_h, \nabla I_hz_h) + (c\nabla u_h, \nabla I_hz_h)_{\mathcal{Q}_h'} \\
&= r_h'(c\nabla u_h, \nabla I_hz_h) + (v, I_hz_h)_{\mathcal{Q}_h}.
\end{align*}
Subtracting these two equalities gives
\begin{align}
\label{eq:ell2c}
(c\nabla [u-u_h],\nabla I_hz_h) &=  -r_h'(c\nabla u_h, \nabla I_hz_h) + r_h(v, I_hz_h).
\end{align}
Now, set $q:=\min(p-1,k_u-2)$. We can write
\begin{align*}
r_h'(c\nabla u_h, \nabla I_hz_h)&= r_h'(\nabla u_h, c\nabla I_hz_h - \Pi_{h,0}c\nabla I_hz_h) \\
&= r_h'(\nabla u_h - \Pi_{h,q}\nabla u, c\nabla I_hz_h - \Pi_{h,0}c\nabla I_hz_h) \\
&\phantom{=-}+ r_h'(\Pi_{h,q}\nabla u, c\nabla I_hz_h - \Pi_{h,0}c\nabla I_hz_h) \\
&= r_h'(\nabla u_h - \Pi_{h,q}\nabla u, c\nabla I_hz_h - \Pi_{h,0}c\nabla I_hz_h)+ r_h'(c\Pi_{h,q}\nabla u, \nabla I_hz_h ) \\
&=: R_1 + R_2,
\end{align*}
where we used C8 for the first and third equality. We can bound $R_1$ as follows: 
\begin{align*}
|R_1| &\leq \|\nabla u_h-\Pi_{h,q}\nabla u\|_0 \| c\nabla I_hz_h - \Pi_{h,0}c\nabla I_hz_h \|_{0} \\
&\phantom{\leq } + |\nabla u_h-\Pi_{h,q}\nabla u|_{\mathcal{Q}_h'} | c\nabla I_hz_h - \Pi_{h,0}c\nabla I_hz_h |_{\mathcal{Q}_h'} \\
&\leq Ch \|\nabla u_h-\Pi_{h,q}\nabla u\|_0 \|z_h\|_{2} \\
&\leq Ch (\|\nabla u_h-\nabla u\|_0 + \|\nabla u-\Pi_{h,q}\nabla u\|_0) \|u-u_h\|_0 \\
&\leq Ch^{\min(p,k_u-1,k_v)+1}(\|u\|_{\min(p+1,k_u)} + \|v\|_{\min(p,k_v)}) \|u-u_h\|_0,
\end{align*}
where we used the Cauchy--Schwarz inequality for the first inequality, Lemma \ref{lem:pre1}, the regularity of $c$, and Lemma \ref{lem:pre3} for the second inequality, the triangle inequality and the regularity assumption for the third inequality, and Theorem \ref{thm:ell1} and Lemma \ref{lem:pre3} for the last inequality. We can also bound $R_2$ as follows:
\begin{align*}
|R_2| &\leq Ch^{p+1}\|c\Pi_{h,q}\nabla u\|_{\mathcal{T}_h,p+1}\|\nabla I_hz_h\|_{\mathcal{T}_h,1} \\
&\leq Ch^{p+1}\|\Pi_{h,q}\nabla u\|_{\mathcal{T}_h,p+1}\|I_hz_h\|_{\mathcal{T}_h,2} \\
&\leq Ch^{p+1}\|\Pi_{h,q}\nabla u\|_{\mathcal{T}_h,p+1}\|z_h\|_{2} \\
&= Ch^{p+1}\|\Pi_{h,q}\nabla u\|_{\mathcal{T}_h,q}\|z_h\|_2 \\
&\leq Ch^{p+1}\|\nabla u\|_{\mathcal{T}_h,q}\|z_h\|_2 \\
&\leq  Ch^{p+1}\|u\|_{\min(p,k_u-1)}\|u-u_h\|_0.
\end{align*}
Here, the first line follows from Lemma \ref{lem:int2} and the fact that $c\Pi_{h,q}\nabla u\in H^{p+1}(\mathcal{T}_h)^3$, the second line follows from the regularity of $c$, the third line follows from Lemma \ref{lem:pre3}, the fifth line follows from Lemma \ref{lem:pre3}, and the last line follows from the regularity assumption. The fourth line follows from the fact that $\Pi_{h,q}\nabla u$ is piecewise polynomial of degree $q$ and therefore $\|\Pi_{h,q}\nabla u\|_{\mathcal{T}_h,p+1}=\|\Pi_{h,q}\nabla u\|_{\mathcal{T}_h,q}$. By combining the bounds on $R_1$ and $R_2$, we then obtain
\begin{align}
|r_h'(c\nabla u_h, \nabla I_hz_h)| &= |R_1+R_2| \leq |R_1| + |R_2| \nonumber \\
&\leq Ch^{\min(p+1,k_u,k_v+1)}(\|u\|_{\min(p+1,k_u)} + \|v\|_{\min(p,k_v)})\|u-u_h\|_0. \label{eq:ell2d}
\end{align}
From Lemma \ref{lem:int1}, Lemma  \ref{lem:pre3}, and the regularity assumption, it also follows that
\begin{align}
\label{eq:ell2e}
|r_h(v, I_hz_h)| &\leq Ch^{\min(p+1,k_v)}\|v\|_{\min(p+1,k_v)}\|I_hz\|_{\mathcal{T}_h,2} \nonumber \\
&\leq Ch^{\min(p+1,k_v)}\|v\|_{\min(p+1,k_v)}\|z_h\|_2 \nonumber \\
&\leq Ch^{\min(p+1,k_v)}\|v\|_{\min(p+1,k_v)}\|u-u_h\|_0.
\end{align}
Combining \ccref{eq:ell2c}, \ccref{eq:ell2d}, and \ccref{eq:ell2e} gives
\begin{align*}
|(c\nabla [u-u_h],\nabla I_hz_h)| &\leq Ch^{\min(p+1,k_u,k_v)}(\|u\|_{\min(p+1,k_u)} + \|v\|_{\min(p+1,k_v)})\|u-u_h\|_0.
\end{align*}
Combining this with \ccref{eq:ell2a} and \ccref{eq:ell2b} then gives \ccref{eq:ell2}.
\end{proof}

These results can be used to prove optimal convergence for the wave equation in a way analogous to \cite[Chapter 4.6]{geevers18b} by replacing $a(u,w)$ by $a_h(u,w):=(c\nabla u,\nabla w)_{\mathcal{Q}_h'}$ and by defining the projection operator $\pi_h$ of \cite{geevers18b} such that $a_h(\pi_hu,w)=(\nabla\cdot c\nabla u,w)_{\mathcal{Q}_h}$ for all $w\in U_h$.

\subsection{Error estimates for the linear elastic case}
So far, we only analyzed the scalar wave equation, but we can obtain error estimates for the elastic wave equations in an analogous way. 

In the linear elastic case, the wave field $\vu:\Omega\times(0,T)\rightarrow\mathbb{R}^3$ is a vector field and \ccref{eq:modela} becomes
\begin{align*}
\rho\partial_t^2\vu &= \nabla\cdot C:\nabla \vu + \vct{f} &&\text{in }\Omega\times(0,T),
\end{align*}
with $[\nabla\cdot C:\nabla \vu]_{i}=\sum_{j,k,l=1}^3 \partial_jC_{jilk}\partial_ku_l$, where $C:\Omega\rightarrow\mathbb{R}^{3\times 3\times 3\times 3}$ is the elastic tensor field with symmetries $C_{ijkl}=C_{jikl}=C_{ijlk}=C_{klij}$ and bounds 
\begin{align*}
c_0\|\tsigma + \tsigma^t \| &\leq \|C:\tsigma\| \leq c_1\|\tsigma + \tsigma^t\| &&\text{for all }\tsigma\in\mathbb{R}^{3\times 3},
\end{align*}
with $c_0$, $c_1$ strictly positive constants and $\|\tsigma\|^2:=\sum_{i,j=1}^3\sigma_{ij}^2$. 

The only part of the error analysis that requires some additional work in this case, is the second inequality of Lemma \ref{lem:pre2}. Instead of $\|\nabla u\|_{\mathcal{Q}_h'} \geq C\|\nabla u\|_0$, we need to show that, if conditions C1-C3, C6, and C7, are satisfied, then
\begin{align}
\label{eq:pre2el}
\|\nabla \vu_h + \nabla\vu_h^t \|_{\mathcal{Q}_h'} \geq C\|\nabla \vu_h + \nabla\vu_h^t\|_0 &&\text{for all }\vu_h\in U_h^3.
\end{align}
This result follows from the fact that $\tilde U$ is finite dimensional and from the relations
\begin{align*}
\|\nabla \vu_h + \nabla\vu_h^t \|^2_{\mathcal{Q}_h'} &= \sum_{e\in\mathcal{T}_h} \frac{|e|}{|\tilde e|} \big\| \ten{J}_e^{-1}\cdot (\tilde\nabla\tilde\vw_e + \tilde\nabla\tilde\vw_e^t) \cdot \ten{J}^{-t}_e \big\|_{\tilde{\mathcal{Q}}'}^2, \\
\|\nabla \vu_h + \nabla\vu_h^t \|^2_{0} &= \sum_{e\in\mathcal{T}_h} \frac{|e|}{|\tilde e|} \big\| \ten{J}_e^{-1}\cdot (\tilde\nabla\tilde\vw_e + \tilde\nabla\tilde\vw_e^t) \cdot \ten{J}^{-t}_e \big\|_{\tilde e}^2, 
\end{align*}
where $\ten{J}_e:=\nabla\phi_e$ is the Jacobian of the element mapping, $\ten{J}_e^{-t}$ denotes the transposed of $\ten{J}_e^{-1}$, $\tilde\vw_e:=\ten{J}_e\cdot(\vu_h\circ\phi_e)\in\tilde U^3$, and $\|\tilde{\tsigma} \|_{\tilde{\mathcal{Q}}'}^2:=\sum_{\tilde\vx\in\tilde{\mathcal{Q}}'} \omega_{\tilde \vx}' \|\tilde{\tsigma}(\tilde\vx)\|^2$.

Using the boundedness of $C$, \ccref{eq:pre2el}, and Korn's inequality, we can show that the bilinear operator for the elastic case $a_h(\vu,\vw):=(C:\nabla\vu,\nabla\vw)_{\mathcal{Q}_h'}$ is still coercive. The other parts of the error analysis are analogous to the scalar case.

\section{Dispersion analysis}
\label{sec:disp}
To test the effect of the new quadrature rules on the accuracy and time step size, we first analyze the dispersion properties of the resulting mass-lumped finite element method along the lines of \cite{geevers18a}. We consider a homogeneous unbounded domain with $\rho=c=1$ and consider plane waves of the form
\begin{align}
u(\vx,t)=e^{\im(\vkappa\cdot\vx - \omega t)}. 
\label{eq:planeWave}
\end{align}
Here, $\im:=\sqrt{-1}$ denotes the imaginary number, $\vkappa$ denotes the wave vector, and $\omega$ denotes the angular velocity. We also let $\lambda=2\pi/|\vkappa|$ denote the wavelength and $c_P=\sqrt{c/\rho}=1$ denote the wave propagation speed. The angular velocity and wave propagation speed satisfy the relation $\omega = |\vkappa| c_P$.

\begin{figure}[h]
\centering
\begin{subfigure}[b]{0.45\textwidth}
  \includegraphics[width=\textwidth]{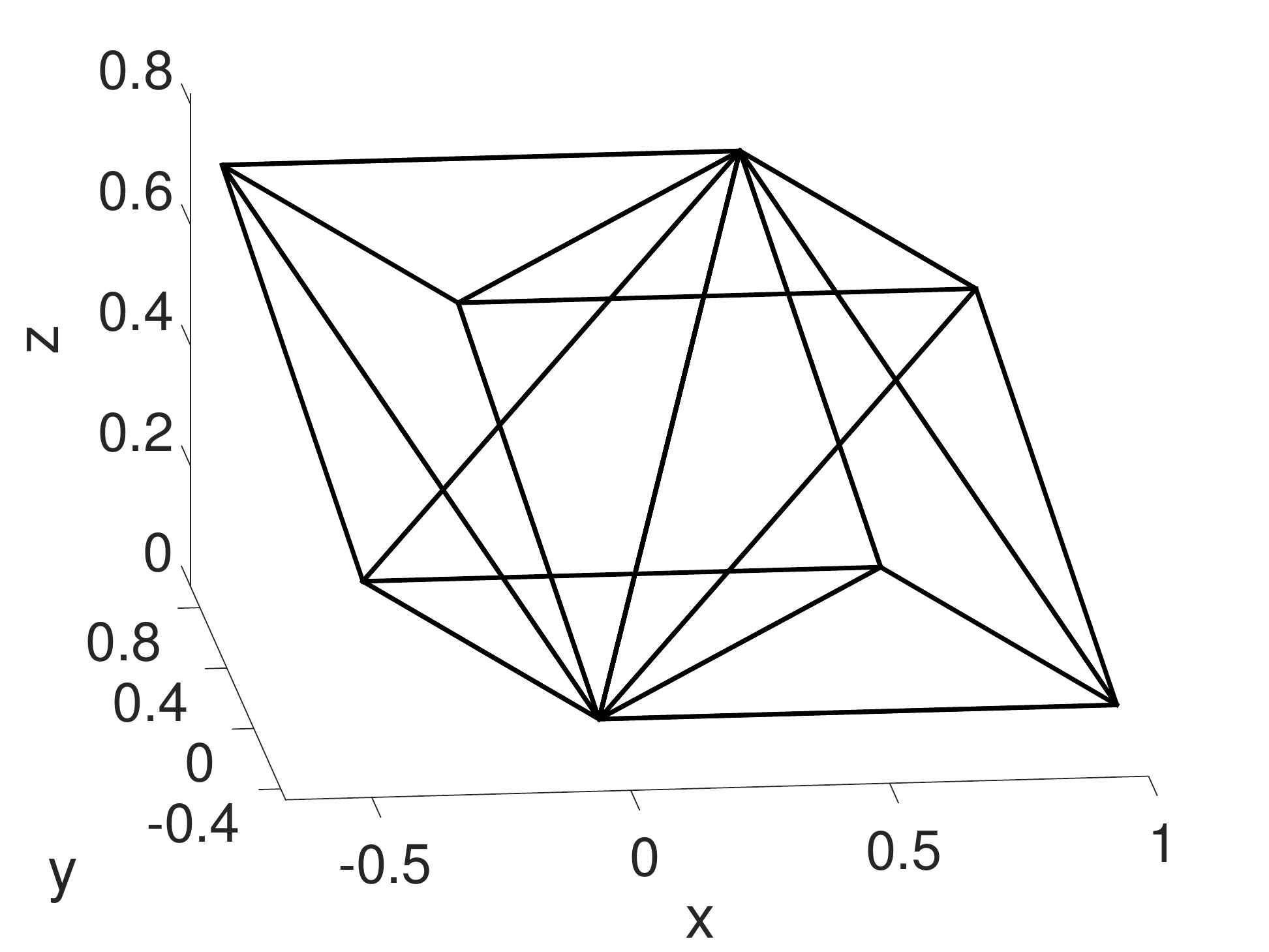}
\end{subfigure} \,\,
\begin{subfigure}[b]{0.45\textwidth}
  \includegraphics[width=\textwidth]{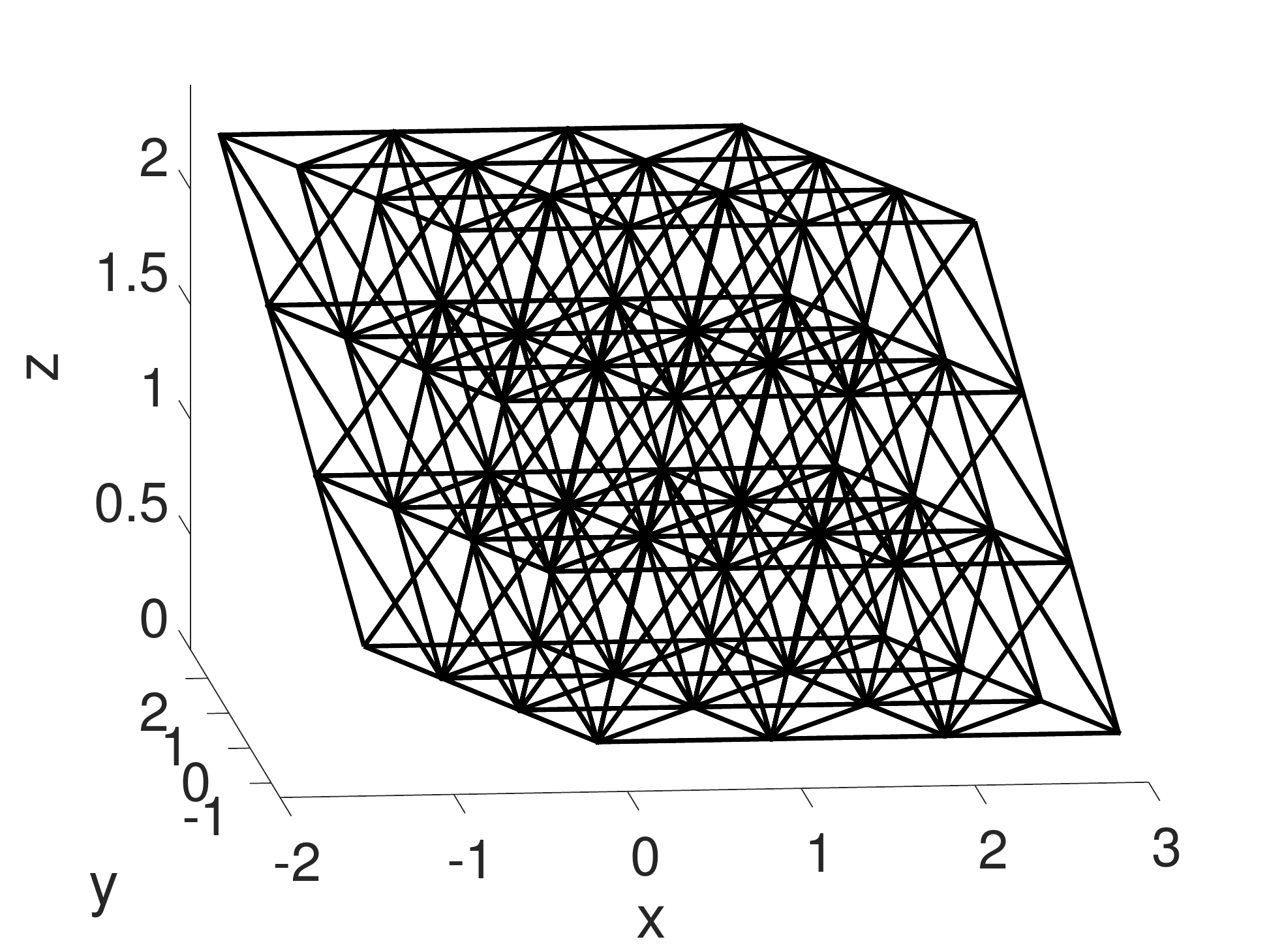}
\end{subfigure}
\caption{Single cell divided into 6 tetrahedra (left) and repetitions of this cell, resulting in the tetragonal disphenoid honeycomb (right).}
\label{fig:cells}
\end{figure}

To analyze the numerical dispersion, we consider a translation-invariant mesh constructed from a repeated cell pattern as illustrated in Figure \ref{fig:cells} and derive the propagation speeds $c_{P,h}$ of the numerical plane waves. The dispersion error $e_{disp}$ is defined as the error in the numerical wave propagation speed: $e_{disp}:=|c_P-c_{P,h}|/c_P$. Since the mesh is constructed from a repeated cell pattern, obtaining the numerical wave propagation speed for a given wave vector requires solving an eigenvalue problem related to only a single cell. 

To construct the translation-invariant mesh, we subdivide the unit cell $[0,1)$ into tetrahedra and repeat this pattern to pack the entire 3D space. We also apply a linear transformation $\vx \rightarrow \ten{T}\cdot\vx$, with $\ten{T}\in\mathbb{R}^{3\times 3}$ and $[\ten{T}\cdot\vx]_i = \sum_{j=1}^3 T_{ij}x_j$, to the mesh in order to obtain more regular tetrahedra. 

Let $\{\vx^{(\Omega_0,i)}\}_{i=1}^{n_0}$ denote all the nodes on $\Omega_0:=\ten{T}\cdot[0,1)^3$, let $\{\vx^{(\Omega_{\vct{k}},i)}\}_{i=1}^{n_0}$ denote the translated nodes on the translated cell $\Omega_{\vct{k}}=\ten{T}\cdot\vct{k}+\Omega_0$, and let $w^{(\Omega_{\vct{k}},i)}$ denote the corresponding nodal basis functions. We define matrices $M^{(\Omega_0)}, A^{(\Omega_0,\Omega_{\vct{k}})}\in\mathbb{R}^{n_0\times n_0}$ as follows:
\begin{align*}
M^{(\Omega_0)}_{ij} &= (\rho w^{(\Omega_0,i)},w^{(\Omega_0,j)})_{\mathcal{Q}_h}, &&i,j=1,\dots,n_0, \\
A^{(\Omega_0,\Omega_{\vct{k}})}_{ij} &= (c\nabla w^{(\Omega_0,i)},\nabla w^{(\Omega_0,j)})_{\mathcal{Q}_h'}, &&i,j=1,\dots,n_0, \vct{k}\in\{-1,0,1\}^3.
\end{align*}
For any wave vector $\vkappa\in\mathbb{R}^3$, we then define the matrix $S^{(\vkappa)}\in \mathbb{R}^{n_0\times n_0}$ as follows:
\begin{align*}
S^{(\vkappa)} &= \left(M^{(\Omega_0)}\right)^{-1}\left(\sum_{\vct{k}\in\{-1,0,1\}^3} e^{\im(\vkappa\cdot \ten{T}\cdot\vct{k})} A^{(\Omega_0,\Omega_{\vct{k}})}\right).
\end{align*}
When using an order-$2K$ Dablain time integration scheme \cite{dablain86}, with time step size $\Delta t$, the angular velocities of the numerical plane waves are given by
\begin{align*}
\omega_h^{(\vkappa,i)} &= \pm \frac{1}{\Delta t} \arccos\left( \sum_{k=0}^K \frac{1}{(2k)!}(-\Delta t^2s_h^{(\vkappa,i)})^k \right),
\end{align*}
where $\{s_{h}^{(\vkappa,i)}\}_{i=1}^{n_0}$ are the eigenvalues of $S^{(\vkappa)}$ \cite{geevers18a}. The numerical wave propagation speeds are given by $c_{P,h}^{(\vkappa,i)}=|\omega_h^{(\vkappa,i)}|/|\vkappa|$. The dispersion error, for a given wavelength $\lambda$, is then given by
\begin{align*}
e_{disp}^{(\lambda)} &:= \sup_{\vkappa\in\mathbb{R}^3, |\vkappa|=2\pi/\lambda} \left( \inf_{i=1,\dots,n_0} \frac{|c_P-c_{P,h}^{(\vkappa,i)}|}{c_P} \right).
\end{align*}

For the dispersion analysis, we consider a mesh of congruent nearly-regular isofacial tetrahedra, known as the tetragonal disphenoid honeycomb. This mesh is obtained by slicing the unit cube $[0,1)^3$ into 6 tetrahedra with the planes $x_1=x_2$, $x_2=x_3$, and $x_1=x_3$, and applying a linear transformation $\vx\rightarrow\ten{T}\cdot\vx$, with
\begin{align*}
\ten{T} &= \begin{bmatrix}
1 & -1/3 & -1/3 \\
0 & \sqrt{8/9} & -\sqrt{2/9} \\
0 & 0 & \sqrt{2/3}
\end{bmatrix}.
\end{align*}
An illustration of this mesh is given in Figure \ref{fig:cells}.

We plot the dispersion error for different elements and quadrature rules against the number of elements per wavelength $N_E :=(\lambda^3/|e|_{av})^{1/3}$, where $|e|_{av}=2\sqrt{3}/27$ denotes the average element volume. We also compute the largest allowed time step size, given by 
\begin{align*}
\Delta t=\sqrt{c_K/s_{h,max}},
\end{align*}
with $c_K$ a constant depending on the order of the time integration scheme ($c_K=4,12,7.57,21.48$ for $K=1,2,3,4$, respectively) and 
\begin{align*}
s_{h,max} &= \sup_{\vct{k}\in\mathcal{K}_0} \max_{i=1,\dots,n_0} s_h^{(\vkappa,i)}
\end{align*}
the largest eigenvalue of the discrete spatial operator, with $\mathcal{K}_0 := \ten{T}^{-t}\cdot [0,2\pi)$ the space of distinct wave vectors. Details on the dispersion analysis can be found in \cite{geevers18a}.

\begin{figure}[h]
\centering
\includegraphics[width=0.6\textwidth]{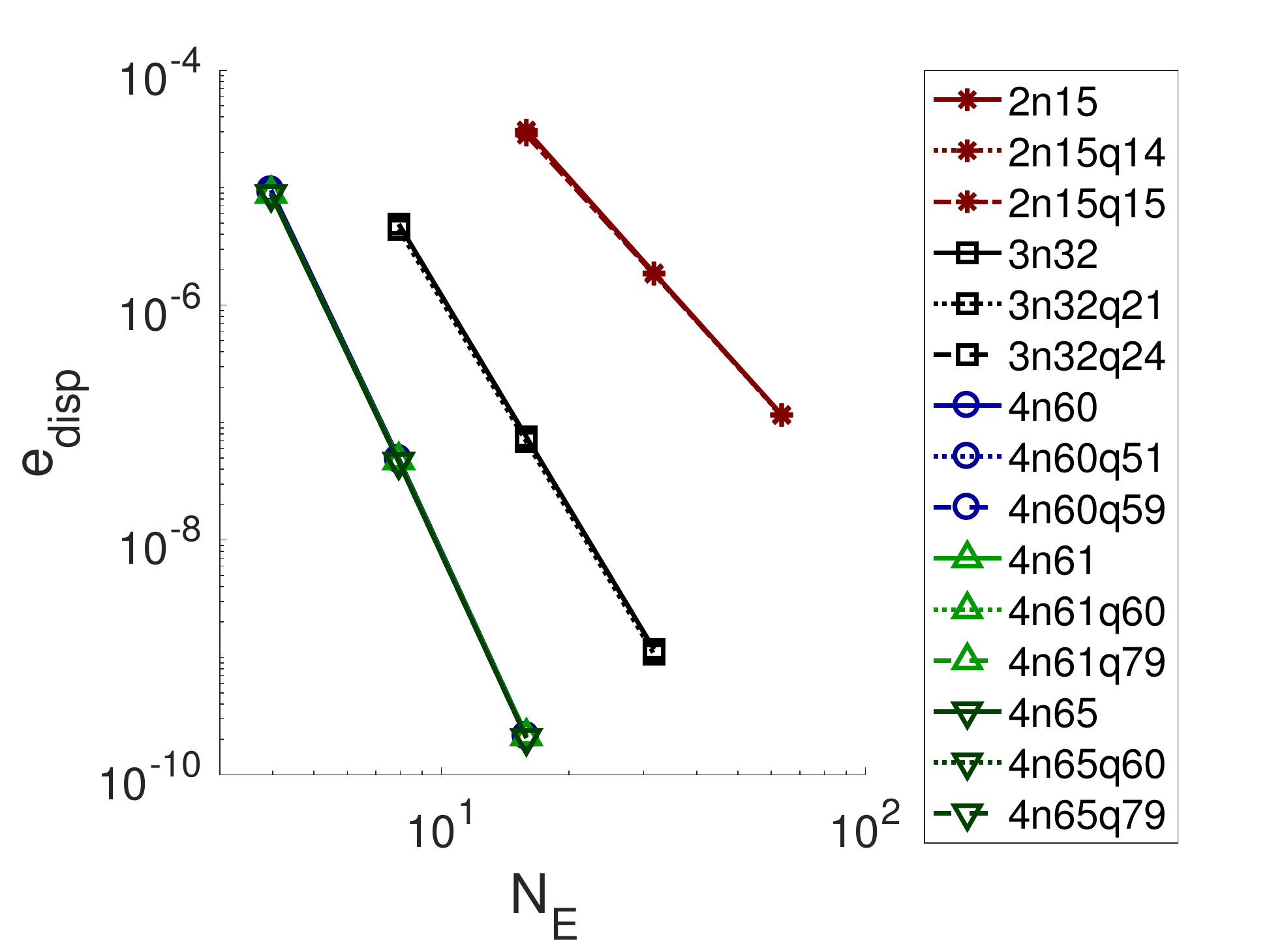}
\caption{Dispersion error $e_{disp}$ versus the number of elements per wave length $N_E$ for different mass-lumped finite element methods. In the legend, [$p$]n[$n$]q[$n'$] denotes the degree-$p$ mass-lumped finite element with $n$ nodes and  $n'$ quadrature points for evaluating the element stiffness matrix. Solid lines correspond to exact integration, dotted lines correspond to quadrature rules presented in this paper, and dashed lines to quadrature rules taken from \cite{zhang09}. Graph 2n15q15 corresponds to the degree-2 method using the mass matrix quadrature rule as stiffness matrix quadrature rule. Graphs of methods with the same polynomial degree are almost identical.}
\label{fig:errDisp}
\end{figure}

\begin{table}[h]
\caption{Dispersion error in terms of the number of elements per wavelength $N_E$, based on extrapolation of the graphs in Figure \ref{fig:errDisp}. The same notation as in the legend of Figure \ref{fig:errDisp} is used. Methods using the new quadrature rules presented in this paper are scripted in bold.}
\label{tab:errDisp}
\begin{center}
\begin{tabular}{l |l ||l |l ||l |l}
method		& $e_{disp}$ 		& method			& $e_{disp}$		 & method 		& $e_{disp}$ 	\\ \hline\hline
2n15 		& $1.89(N_E)^{-4}$ 	& 4n60			& $0.865(N_E)^{-8}$ & 4n65 			& $0.825(N_E)^{-8}$ \\
2n15q14		& $1.86(N_E)^{-4}$ 	& \textbf{4n60q51} 	& $0.842(N_E)^{-8}$ & \textbf{4n65q60} 	& $0.827(N_E)^{-8}$ \\ 
2n15q15		& $1.88(N_E)^{-4}$ 	& 4n60q59 		& $0.861(N_E)^{-8}$ & 4n65q79 		& $0.825(N_E)^{-8}$ \\ \hline
3n32			& $1.20(N_E)^{-6}$  & 4n61			& $0.854(N_E)^{-8}$ &&\\
\textbf{3n32q21}& $1.09(N_E)^{-6}$	& \textbf{4n61q60} 	& $0.854(N_E)^{-8}$ &&\\ 
3n32q24		& $1.19(N_E)^{-6}$	& 4n61q79 		& $0.851(N_E)^{-8}$ &&			
\end{tabular}
\end{center}
\end{table}

\begin{table}[h]
\caption{Largest allowed time step size for different mass-lumped finite element methods. The same notation as in the legend of Figure \ref{fig:errDisp} is used. Methods using the new quadrature rules presented in this paper are scripted in bold.}
\label{tab:dt}
\begin{center}
\begin{tabular}{l |l ||l |l ||l |l}
method		& $\Delta t$ 		& method			& $\Delta t$		 & method 		& $\Delta t$ 	\\ \hline\hline
2n15 		& $0.291$ 	& 4n60			& $0.0508$ & 4n65 			& $0.0932$ \\
2n15q14		& $0.280$ 	& \textbf{4n60q51} 	& $0.0580$ & \textbf{4n65q60} 	& $0.0936$ \\ 
2n15q15		& $0.181$ 	& 4n60q59 		& $0.0578$ & 4n65q79 		& $0.0935$ \\ \hline
3n32			& $0.128$		& 4n61			& $0.0721$ &&\\
\textbf{3n32q21}& $0.136$	& \textbf{4n61q60} 	& $0.0796$ &&\\ 
3n32q24		& $0.135$		& 4n61q79 		& $0.0792$ &&	
\end{tabular}
\end{center}
\end{table}

We test the degree-$p$ mass-lumped finite element methods presented in \cite{geevers18b} and given in Table \ref{tab:MLtet} using exact stiffness matrix evaluation and using the quadrature rules presented in this paper and quadrature rules that are accurate up to degree $p+p'-2$ from \cite{zhang09}, with $p'$ the highest polynomial degree of the enriched element space. For the degree-2 method, we also test using the quadrature rule of the mass matrix for evaluating the stiffness matrix. We combine each method with an order-$2p$ Dablain time integration scheme.

The dispersion error versus the number of elements per wavelength is shown in Figure \ref{fig:errDisp} and extrapolations of these graphs are given in Table \ref{tab:errDisp}. The figure and table show that the dispersion error of all methods converges with order $2p$, which is typical for eigenvalue and dispersion errors of symmetry-conserving finite element approximations, see for example, \cite{boffi10} and the references therein. The figure and table also show that there is hardly any difference in accuracy between the methods of the same degree and that the methods using a quadrature rule to evaluate the stiffness matrix have a nearly the same or even slightly smaller dispersion error than the methods using exact integration. 

The largest allowed time step size for each method is given in Table \ref{tab:dt}. The table shows that the quadrature rules for the stiffness matrix tested here hardly affect the largest allowed time step size, except for the degree-2 15-point mass matrix quadrature rule, which reduces the allowed time step size by more than a factor 1.5. For the other quadrature rules, the largest allowed time step size remains nearly the same or becomes even slightly larger.

If the stiffness matrix is evaluated with a quadrature rule and the resulting accuracy and time step size remain nearly the same, then the number of computations to obtain a given accuracy mainly depends on the number of quadrature points. Since the quadrature rules presented in this paper satisfy these properties and require less quadrature points than those currently available in the literature, they can result in a reduction of the computational cost proportionate to the reduction in number of quadrature points.

\section{Numerical tests}
\label{sec:tests}

\subsection{Algorithms for computing the element stiffness matrices}
\label{sec:alg}
Before we present the numerical tests, we first briefly describe the algorithms for computing the element stiffness matrices. In particular, we show how we efficiently compute the element stiffness matrix-vector products on the fly. We do not store the matrices, since this requires storing and fetching significantly more data, and since it was shown in \cite{mulder16} that an on-the-fly approach is more efficient for higher-degree elements.

To describe the algorithms, let $e\in\mathcal{T}_h$ be an arbitrary element. We introduce the following notation.
\begin{itemize}
  \item $\{\tilde\vx_i\}_{i=1}^n=\tilde{\mathcal{Q}}$: nodes on reference element $\tilde e$. Nodes of the different mass-lumped elements can be found in \cite{geevers18b}.
  \item $\tilde{w}_i$: nodal basis function corresponding to $\tilde\vx_i$.
  \item $w_i^{(e)}:=\tilde{w}_i\circ\phi_e^{-1}$: nodal basis function of the physical element.
  \item $\{\tilde\vx_i'\}_{i=1}^{n'}=\tilde{\mathcal{Q}}'$: quadrature points for the stiffness matrix on reference element $\tilde e$. Quadrature rules for the different elements are given in Section \ref{sec:quadRules}.
  \item $\tilde\omega_i'$: quadrature weight corresponding to $\tilde\vx_i'$.
  \item $A^{(e)}\in\mathbb{R}^{n\times n}$: the element stiffness matrix.
  \item $u^{(e)}$: the wave field on $e$.
  \item $\vvu^{(e)}\in\mathbb{R}^n$: the wave field at the nodes on $e$. 
\end{itemize} 
When using exact integration, the stiffness matrix-vector product $\vvv^{(e)}:=A^{(e)}\vvu^{(e)}\in\mathbb{R}^n$ is given by
\begin{align}
\label{eq:Aue}
[A^{(e)}\vvu^{(e)}]_i &= \int_e c\nabla w_i^{(e)}\cdot \nabla u^{(e)} \;dx,
\end{align}
for $i=1,\dots,n$. After rewriting the integral as an integral over the reference element, this becomes
\begin{align}
\label{eq:AueRef}
[A^{(e)}\vvu^{(e)}]_i &= \int_{\tilde e} \tilde\nabla\tilde{w}_i\cdot \tilde{\ten{c}}^{(e)} \cdot \tilde\nabla \tilde{u}^{(e)} \;d\tilde x, 
\end{align}
where $\tilde{u}^{(e)}:=u^{(e)}\circ\phi_e$, $\tilde\nabla$ is the gradient operator in reference coordinates, and $\tilde{\ten{c}}^{(e)}:=(c\circ\phi_e)\frac{|e|}{|\tilde e|}\ten{J}_e^{-t}\cdot\ten{J}_e^{-1}$ is a tensor field, with $\ten{J}_e:=\nabla\phi_e$ the Jacobian of the element mapping and $\ten{J}_e^{-t}$ the transposed of $\ten{J}_e^{-1}$. When $c$ is constant within each element, then $\tilde{\ten c}^{(e)}$ is also constant and we can compute \ccref{eq:AueRef} using the algorithm of  \cite{mulder16}:
\begin{align}
\label{eq:Au1}
[A^{(e)}\vvu^{(e)}]_i &= \sum_{i_D,j_D=1}^3 \tilde{c}^{(e)}_{i_D,j_D} \left(\sum_{j=1}^n B^{(i_D,j_D)}_{ij}\vvu^{(e)}_j\right),
\end{align}
where $B^{(i_D,j_D)}\in\mathbb{R}^{n\times n}$ are precomputed matrices, given by
\begin{align*}
B^{(i_D,j_D)}_{ij} &= \int_{\tilde e} (\tilde\partial_{i_D}\tilde{w}_i)(\tilde\partial_{j_D}\tilde{w}_j) \;d\tilde x,
\end{align*}
for $i_D,j_D=1,2,3$, $i,j=1,\dots,n$, with $\tilde\partial_{i_D}$ the derivative in reference coordinate $i_D$. We can reduce the number of computations in \ccref{eq:Au1} using the fact that $\tilde{\ten c}^{(e)}$ is symmetric:
\begin{align}
[A^{(e)}\vvu^{(e)}]_i &= \sum_{i_D=1}^3\sum_{j_D=1}^{i_D} \tilde{c}^{(e)}_{i_D,j_D} \left(\sum_{j=1}^n \hat{B}^{(i_D,j_D)}_{ij}\vvu^{(e)}_j\right),
\end{align}
where $\hat{B}^{(i_D,j_D)}:=B^{(i_D,j_D)}+B^{(j_D,i_D)}$ if $i_D\neq j_D$ and $\hat{B}^{(i_D,j_D)}:=B^{(i_D,j_D)}$ when $i_D=j_D$. The complete algorithm can then be described as follows:
\begin{enumerate}[{A}1.]
  \item Compute $\epsilon^{(i_D,j_D)}\in\mathbb{R}^n$ for $i_D=1,2,3$, $j_D\leq i_D$: 
  \begin{align*}
  \epsilon^{(i_D,j_D)}_i=\sum_{j=1}^n \hat{B}^{(i_D,j_D)}_{ij}\vvu^{(e)}_j.
  \end{align*}
  \item Compute $A^{(e)}\vvu^{(e)}\in\mathbb{R}^n$: 
  \begin{align*}
  [A^{(e)}\vvu^{(e)}]_i = \sum_{i_D=1}^3\sum_{j_D=1}^{i_D} \tilde{c}^{(e)}_{i_D,j_D}\epsilon^{(i_D,j_D)}_i.
  \end{align*}
\end{enumerate}
The computational work is dominated by the first step where 6 matrix-vector products with matrices of size $n\times n$ are computed.

Alternatively, we can compute $A^{(e)}\vvu^{(e)}$ by evaluating the integrals with a quadrature rule. Equation \ccref{eq:AueRef} then becomes
\begin{align}
\label{eq:AuNeRef}
[A^{(e)}\vvu^{(e)}]_i &= \sum_{k=1}^{n'} \tilde\nabla\tilde{w}_i(\tilde{\vx}_k')\cdot \tilde{\ten{c}}^{(e,k)} \cdot \tilde\nabla \tilde{u}^{(e)}(\tilde{\vx}_k'), 
\end{align}
where $\tilde{\ten{c}}^{(e,k)}:= \tilde{\omega}_k'\tilde{\ten{c}}^{(e)}(\tilde{\vx}_k') \in\mathbb{R}^{3\times 3}$. We can compute this as follows:
\begin{align}
[A^{(e)}\vvu^{(e)}]_i &=  \sum_{k=1}^{n'}\sum_{i_D=1}^3  D^{(i_D)}_{ki} \left(\sum_{j_D=1}^3 \tilde{c}^{(e,k)}_{i_D,j_D} \left(\sum_{j=1}^nD^{(j_D)}_{kj}\vvu^{(e)}_j \right)\right),
\end{align}
where $D^{(i_D)}\in\mathbb{R}^{n'\times n}$ are precomputed matrices, given by
\begin{align*}
D^{(i_D)}_{ki} &= (\tilde\partial_{i_D}\tilde{w}_{i})(\tilde{\vx}_k').
\end{align*}
The complete algorithm can be described as follows:
\begin{enumerate}[{B}1.]
  \item Compute $\epsilon^{(j_D)}\in\mathbb{R}^{n'}$ for $j_D=1,2,3$:
  \begin{align*}
  \epsilon^{(j_D)}_k &= \sum_{j=1}^nD^{(j_D)}_{kj}\vvu^{(e)}_j.
  \end{align*}
  \item Compute $\sigma^{(i_D)}\in\mathbb{R}^{n'}$ for $i_D=1,2,3$:
  \begin{align*}
  \sigma^{(i_D)}_k &= \sum_{j_D=1}^3 \tilde{c}^{(e,k)}_{i_D,j_D} \epsilon^{(j_D)}_k.
  \end{align*}
  \item Compute $A^{(e)}\vvu^{(e)}\in\mathbb{R}^n$:
  \begin{align*}
  [A^{(e)}\vvu^{(e)}]_i &= \sum_{i_D=1}^3 \left(\sum_{k=1}^{n'}  D^{(i_D)}_{ki} \sigma^{(i_D)}_k  \right).
  \end{align*}
\end{enumerate}
The computational work for this algorithm is dominated by the first and third step, which both require 3 matrix-vector products with matrices of size $n'\times n$, so 6 of these matrix-vector products in total. Since $n'<n$ for all the quadrature rules presented in this paper, this number of computations is smaller than for the previous algorithm, although only slightly. However, as we will show next, this quadrature-based algorithm is significantly more efficient than the exact-integral algorithm in case of linear elasticity. Moreover, this quadrature-based algorithm also works if $c$ varies within the element. 

In case of linear elasticity, the wave field $\vu:\Omega\times(0,T)\rightarrow \mathbb{R}^3$ is a vector field and the term $c\nabla u$ becomes the stress tensor $C:\nabla\vu$, with $C\in\mathbb{R}^{3\times 3\times 3\times 3}$ the order-four elasticity tensor with symmetries $C_{ijkl}=C_{jikl}=C_{ijlk}=C_{klij}$ and $[C:\nabla\vu]_{ij}:=\sum_{k,l=1}^3C_{ijkl}\partial_lu_k$. The vector $\vvu^{(e)}\in\mathbb{R}^{3n}$ can in this case be written as a concatenation of three vectors $\vvu^{(e,1)}, \vvu^{(e,2)}, \vvu^{(e,3)}\in\mathbb{R}^n$, where $\vvu^{(e,i)}$ is the wave field component $u_{i}$ at the nodes on $e$. The parameter $\tilde{\ten c}^{(e)}$ becomes $\tilde{C}^{(e)}:=\frac{|e|}{|\tilde e|}\ten{J}_e^{-t}\cdot(C\circ\phi_e)\cdot\ten{J}_e^{-1}$, where $[\ten{J}_e^{-t}\cdot C \cdot\ten{J}_e^{-1}]_{ijkl}=\sum_{p,q=1}^3[\ten{J}_e^{-t}]_{ip}C_{pjkq}[\ten{J}_e^{-1}]_{ql}$, and $\tilde{\ten c}^{(e,k)}$ becomes $\tilde{C}^{(e,k)}:=\tilde{\omega}_k'\tilde{C}^{(e)}(\tilde{\vx}_k')$. The algorithm for computing the element stiffness matrix-vector product using exact integration then becomes
\begin{enumerate}[{A}1*.]
  \item Compute $\epsilon^{(i_D,j_D,j_V)}\in\mathbb{R}^n$ for $i_D,j_D,j_V=1,2,3$: 
  \begin{align*}
  \epsilon^{(i_D,j_D,j_V)}_i=\sum_{j=1}^n {B}^{(i_D,j_D)}_{ij}\vvu^{(e,j_V)}_j.
  \end{align*}
  \item Define $\vvv^{(e)}:=A^{(e)}\vvu^{(e)}$. Compute $\vvv^{(e,i_V)}\in\mathbb{R}^n$ for $i_V=1,2,3$: 
  \begin{align*}
  \vvv^{(e,i_V)}_i = \sum_{i_D,j_D,j_V=1}^3 \tilde{C}^{(e)}_{i_D,i_V,j_V,j_D}\epsilon^{(i_D,j_D,j_V)}_i.
  \end{align*}
\end{enumerate}
The computational work is again dominated by the first step, which now requires 27 matrix-vector products with matrices of size $n\times n$. 

When using a quadrature rule, the algorithm becomes
\begin{enumerate}[{B}1*.]
  \item Compute $\epsilon^{(j_D,j_V)}\in\mathbb{R}^{n'}$ for $j_D,j_V=1,2,3$:
  \begin{align*}
  \epsilon^{(j_D,j_V)}_k &= \sum_{j=1}^nD^{(j_D)}_{kj}\vvu^{(e,j_V)}_j.
  \end{align*}
  \item Compute $\sigma^{(i_D,i_V)}\in\mathbb{R}^{n'}$ for $i_D,i_V=1,2,3$:
  \begin{align*}
  \sigma^{(i_D,i_V)}_k &= \sum_{j_D,j_V=1}^3 \tilde{C}^{(e,k)}_{i_D,i_V,j_V,j_D} \epsilon^{(j_D,j_V)}_k.
  \end{align*}
  \item Define $\vvv^{(e)}:=A^{(e)}\vvu^{(e)}$. Compute $\vvv^{(e,i_V)}\in\mathbb{R}^n$ for $i_V=1,2,3$: 
  \begin{align*}
  \vvv^{(e,i_V)}_i &= \sum_{i_D=1}^3 \left(\sum_{k=1}^{n'}  D^{(i_D)}_{ki} \sigma^{(i_D,i_V)}_k  \right).
  \end{align*}
\end{enumerate}
The computational work for this algorithm is dominated again by the first and third step, which now both require 9 matrix-vector products with matrices of size $n'\times n$, so 18 of these matrix-vector products in total. The number of computations is therefore reduced by more than a factor 1.5 when compared to the algorithm based on exact integration. Furthermore, the quadrature-based algorithm can also handle tensor fields $C$ that vary within the element.

Both algorithms can be slightly improved by exploiting the fact that the rows and columns of the matrices $B^{(i_D,j_D)}$ and the columns of matrices $D^{(i_D)}$ sum to zero. Furthermore, in case of isotropic elasticity, steps A2* and B2* can be computed more efficiently by exploiting the simple structure of the elasticity tensor $C$.

In the next subsections, we demonstrate the superiority of the quadrature-based algorithm for the case of non-constant parameters and linear elasticity.

\subsection{Acoustic wave on a heterogeneous domain}
We first test the methods for an acoustic wave propagation problem with a heterogeneous domain. The acoustic wave equation is given by
\begin{align}
\label{eq:modelAc}
\frac{1}{\rho c^2}\partial_t^2 p &= \nabla\cdot \frac{1}{\rho}\nabla p, &&\text{in }\Omega\times(0,T),
\end{align}
with spatial domain $\Omega\subset\mathbb{R}^3$, time interval $(0,T)$, pressure field $p:\Omega\times(0,T)\rightarrow \mathbb{R}$, mass density $\rho:\Omega\rightarrow\mathbb{R}$, and acoustic wave speed $c:\Omega\rightarrow\mathbb{R}$. We choose $\Omega:=(-L_1,L_1)\times(-L_2,L_2)\times(-L_3,L_3)$ and impose zero Neumann boundary conditions. 

To construct an analytic solution, let $X_i:=x_i+\frac{a_i}{m_i}\cos(m_ix_i)$, for $i=1,2,3$, be distorted coordinates, with $m_i:=\frac12\pi/L_i$ and $a_i\in[0,1)$, and define $g_i:=\partial_iX_i=1-a_i\sin(m_ix_i)$. Also let $\rho_0\in\mathbb{R}$ be the average mass density, $c_0\in\mathbb{R}$ the average wave speed, $\vct{k}\in\mathbb{R}^3$ the wave vector, and $\omega:=c_0|\vct{k}|$ the angular velocity, and let parameters $\rho$ and $c$ be given by
\begin{align*}
\rho(\vx):=\rho_0g_1(x_1)g_2(x_2)g_3(x_3), \qquad c(\vx):=c_0\sqrt{\frac{k_1^2+k_2^2+k_3^2}{k_1^2g_1^2(x_1)+k_2^2g_2^2(x_2)+k_3^2g_3^2(x_3)}}.
\end{align*}
Then the standing wave, given by
\begin{align*}
p(\vx,t)=\cos(\omega t)\sin(k_1X_1)\sin(k_2X_2)\sin(k_3X_3),
\end{align*}
is a solution of \ccref{eq:modelAc} that satisfies the zero Neumann boundary conditions.

Now, set $L_i=1\;$km, $a_i=0.2$, $k_i=3m_i$, for $i=1,2,3$, and $c_0=2\;$km/s, $\rho_0 = 2\;$g/cm$^3$. To test the numerical methods, we use $p(\vx,0)$ and $\partial_tp(\vx,0)$ as initial conditions. We test on multiple unstructured meshes and simulate in time using a fourth-order time-stepping scheme \cite{dablain86} with time step size $\Delta t=0.99\Delta t_{\max}$, where $\Delta t_{\max} := \sqrt{12/\sigma_{\max}}$ is the largest allowed time step size \cite{geevers18a} and $\sigma_{\max}$ denotes the largest eigenvalue of the spatial operator, which is computed up to four decimals using power iteration. The root mean square (RMS) error is computed after two time oscillations, so at $T=4\pi/\omega\approx 0.7698\;$s.

\begin{table}[h]
\caption{Power-law fits of the left graphs of Figures \ref{fig:rmsAc1} and \ref{fig:rmsAc2}. Convergence rates are given in bold.}
\label{tab:errFit}
\begin{center}
\begin{tabular}{l |c |c}
				& \multicolumn{2}{c}{RMS error} \\ \hline
Method			& Figure \ref{fig:rmsAc1} 		& Figure \ref{fig:rmsAc2} \\ \hline\hline
2,4 15	 		& $(2.9\times 10^2)N^{(-1/3 \times \mathbf{3.2})}$	& $(1.2\times 10^1)N^{(-1/3 \mathbf{\times 2.4})}$ \\ 
3,4 32  			& $(1.9\times 10^3)N^{(-1/3 \times \mathbf{4.1})}$	& $(2.8\times 10^0)N^{(-1/3 \mathbf{\times 2.1})}$ \\  
4,4 60 			& $(5.8\times 10^4)N^{(-1/3 \times \mathbf{5.3})}$	& $(5.3\times 10^0)N^{(-1/3 \mathbf{\times 2.1})}$ \\  
4,4 61 			& $(7.9\times 10^4)N^{(-1/3 \times \mathbf{5.3})}$	& $(5.4\times 10^0)N^{(-1/3 \mathbf{\times 2.1})}$ \\ 
4,4 65		 	& $(7.3\times 10^4)N^{(-1/3 \times \mathbf{5.3})}$	& $(5.7\times 10^0)N^{(-1/3 \mathbf{\times 2.1})}$ 
\end{tabular}
\end{center}
\end{table}

\begin{figure}[h]
\centering
\begin{subfigure}[b]{0.45\textwidth}
  \includegraphics[width=\textwidth]{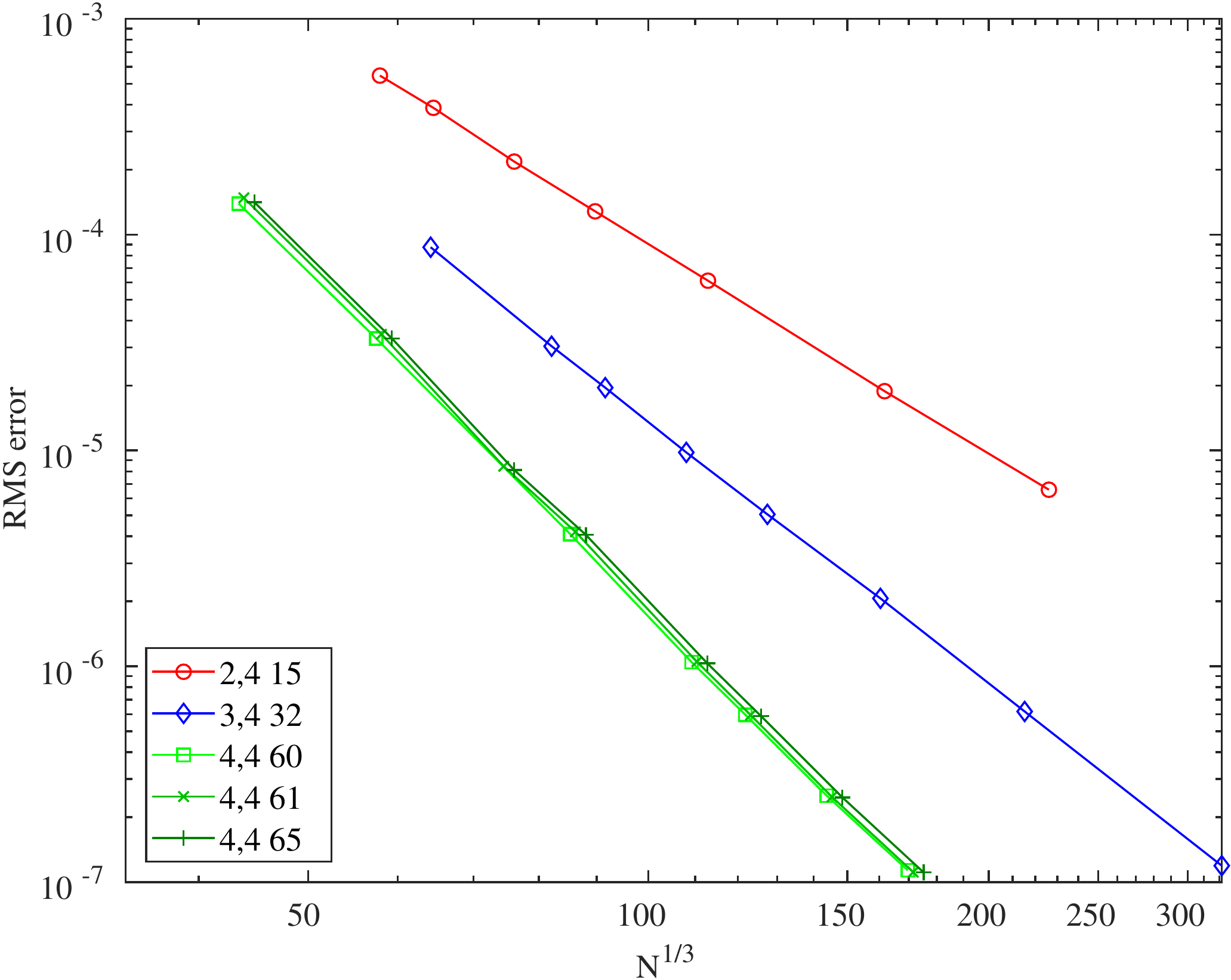}
\end{subfigure} \,\,
\begin{subfigure}[b]{0.45\textwidth}
  \includegraphics[width=\textwidth]{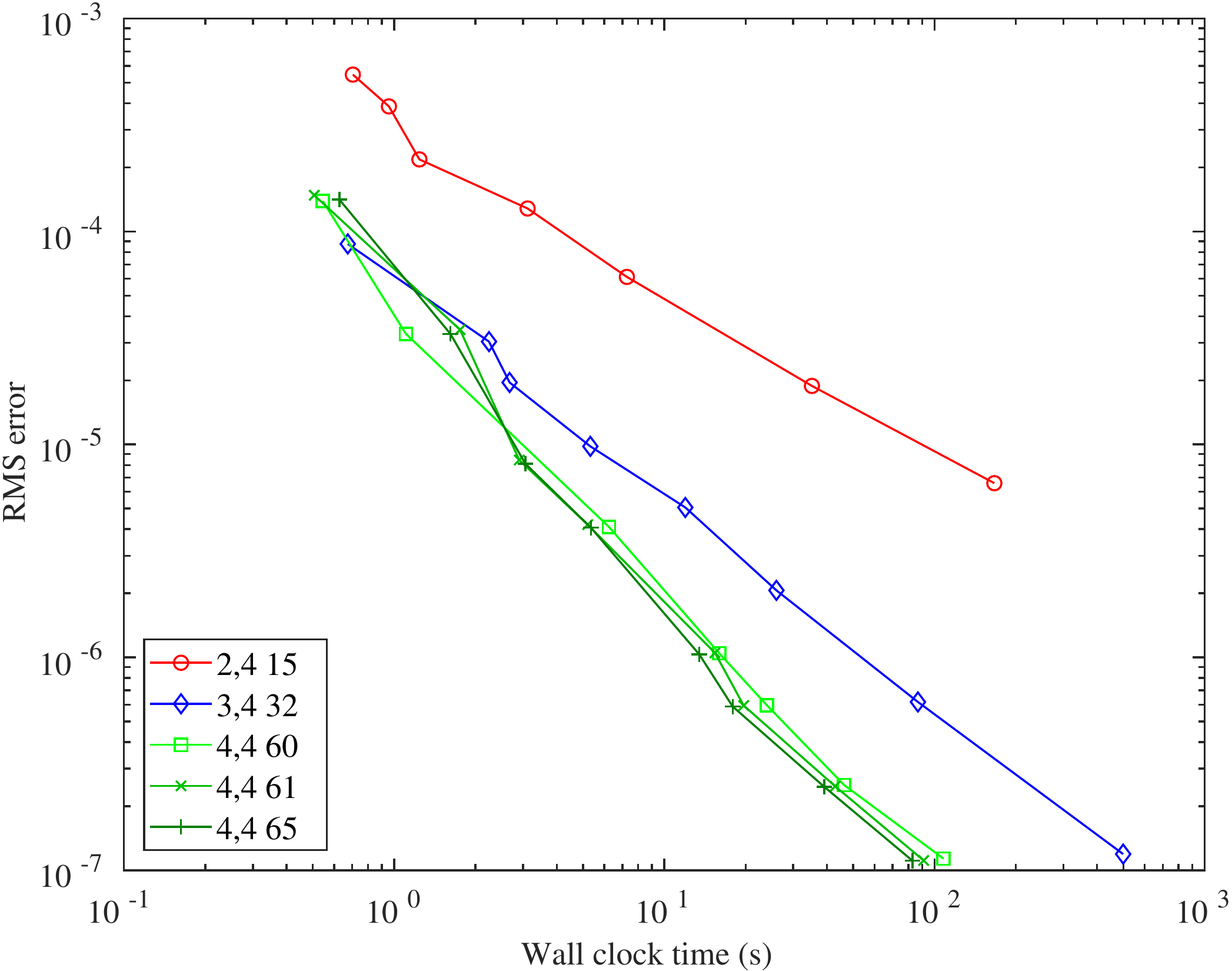}
\end{subfigure}
\caption{RMS errors for the acoustic test case as a function of the cube root of the number of degrees of freedom (left) and as a function of the wall clock time (right). In the legend, [$p,K\;n$] refers to the element of degree $p$ with $n$ nodes, combined with an order-$K$ time-stepping scheme. The element stiffness matrices were evaluated using a quadrature rule.}
\label{fig:rmsAc1}
\end{figure}

Figure \ref{fig:rmsAc1} shows the RMS error plotted against the cube root of the number of degrees of freedom $N$ and the wall-clock time for the mass-lumped tetrahedral element methods using the quadrature-based algorithm for the stiffness matrix as discussed in the previous subsection. The simulations shown here were performed with an OpenMP implementation on 24 cores of two Intel\textsuperscript{\textregistered{}} Xeon\textsuperscript{\textregistered{}} E5-2680 v3 CPUs running at 2.50GHz. Power-law fits of the left graph are also shown in Table \ref{tab:errFit}. This graph shows optimal convergence rates of order $p+1$ and thereby confirms the error estimates of Section \ref{sec:accuracy}. In particular, it confirms that optimal convergence rates are maintained, even though the spatial parameters $\rho$, $c$ vary within the element.

\begin{figure}[h]
\centering
\begin{subfigure}[b]{0.45\textwidth}
  \includegraphics[width=\textwidth]{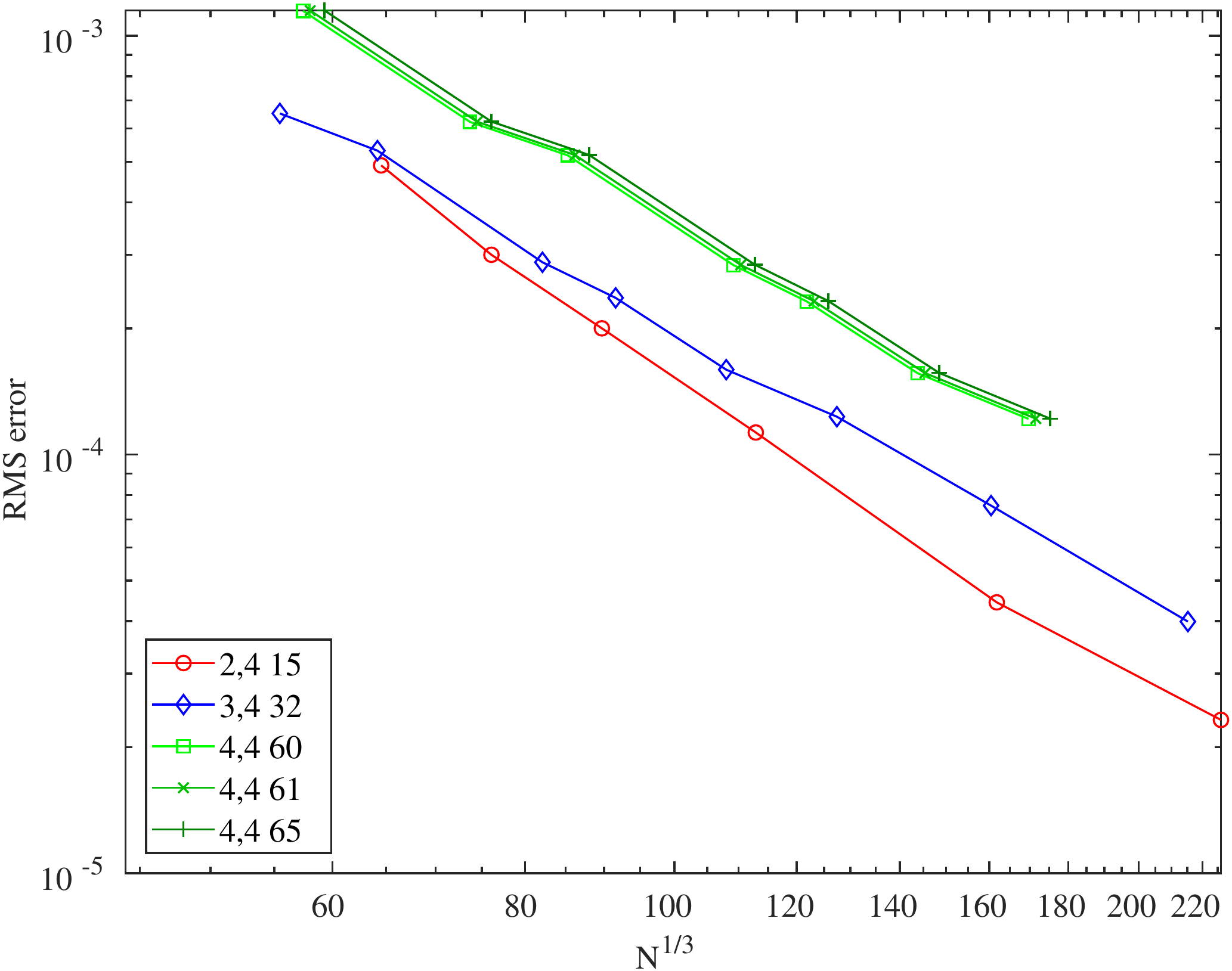}
\end{subfigure} \,\,
\begin{subfigure}[b]{0.45\textwidth}
  \includegraphics[width=\textwidth]{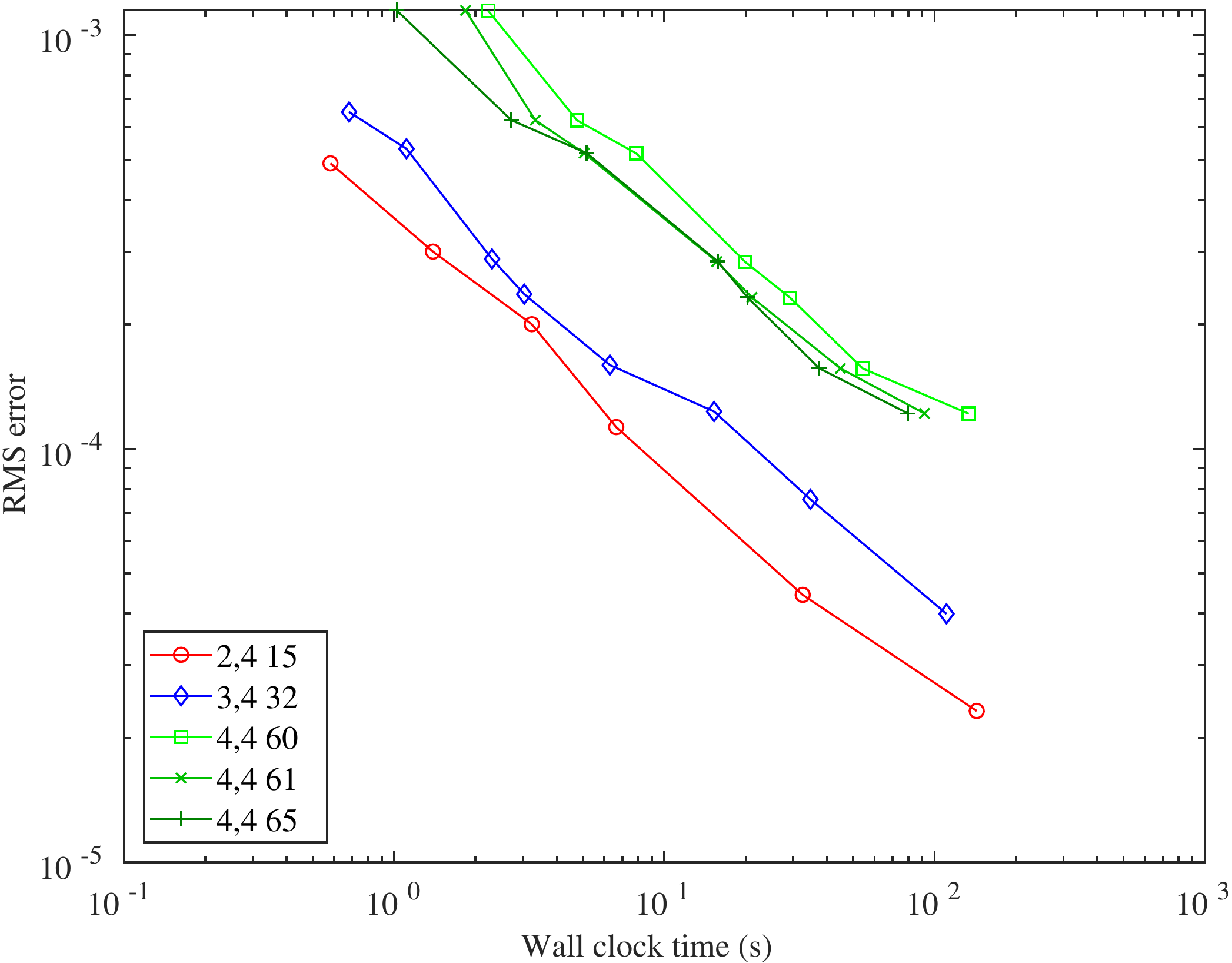}
\end{subfigure}
\caption{Same as Figure \ref{fig:rmsAc1}, but using exact integration to evaluate the stiffness matrix and using a piecewise-constant approximation of the mass density $\rho$. All methods only converge with second order due to the parameter approximation and higher-degree methods only result in more degrees of freedom and computation time.}
\label{fig:rmsAc2}
\end{figure}

Figure \ref{fig:rmsAc2} shows the same as Figure \ref{fig:rmsAc1} for the methods using exact integration to evaluate the stiffness matrix and using a piecewise constant approximation of the mass density $\rho$. Power-law fits of the left graph are again given in Table \ref{tab:errFit}. The graph shows that, due to the piecewise constant approximation, only second-order convergence rates are obtained. The higher-degree elements now only result in more computations per element, without any significant gain in accuracy. When comparing with Figure \ref{fig:rmsAc1}, it follows that the quadrature-based approach is much more efficient than using exact integration with piecewise-constant parameter approximations.

\subsection{Elastic wave on a homogeneous domain}
We also test the methods for an elastic wave propagation problem on a homogeneous domain. The elastic wave equations are given by
\begin{align*}
\rho\partial_t^2\vu &= \nabla\cdot C:\nabla\vu + \vct{f}, &&\text{in }\Omega\times (T_0,T_1),
\end{align*}
with $\vu:\Omega\times(T_0,T_1)\rightarrow\mathbb{R}^3$ the displacement field,  $\vct{f}:\Omega\times(T_0,T_1)\rightarrow\mathbb{R}^3$ the force field, $\rho:\Omega\rightarrow\mathbb{R}$ the mass density, and $C:\Omega\rightarrow\mathbb{R}^{3\times 3\times 3\times 3}$ the elasticity tensor. We consider an isotropic elastic medium, so $C:\nabla\vu=\lambda(\nabla\cdot\vu)\ten{I} + \mu(\nabla\vu+\nabla\vu^t)$, with $\ten{I}\in\mathbb{R}^{3\times 3}$ the identity tensor, $\nabla\vu^t$ the transposed of $\nabla\vu$, and $\lambda,\mu:\Omega\rightarrow\mathbb{R}$ the Lam\'e parameters.

We choose domain $\Omega=[-2,2]\times[-1,1]\times[0,2]\;$km$^3$ with zero Neumann boundary conditions, and set the parameters with a constant mass density $\rho=2\;$g/cm$^3$, primary wave velocity $v_P:=\sqrt{(\lambda+2\mu)/\rho}=2\;$km/s, and secondary/shear wave velocity $w_S:=\sqrt{\mu/\rho}=1.2\;$km/s. A unit vertical force source with a 7-Hz Ricker-wavelet is placed at $\vx_{src}:=(0,0,1000)\;$m and receivers are placed between $x_r=-1375\;$m and $x_r=1375\;$m with a 50-m interval at $y_r=200\;$m and $z_r=800\;$m. The exact solution can be found in \cite{aki80}. The simulation time is chosen such that reflections caused by the boundary conditions do not reach the receivers.

\begin{figure}[h]
\centering
\begin{subfigure}[b]{0.45\textwidth}
  \includegraphics[width=\textwidth]{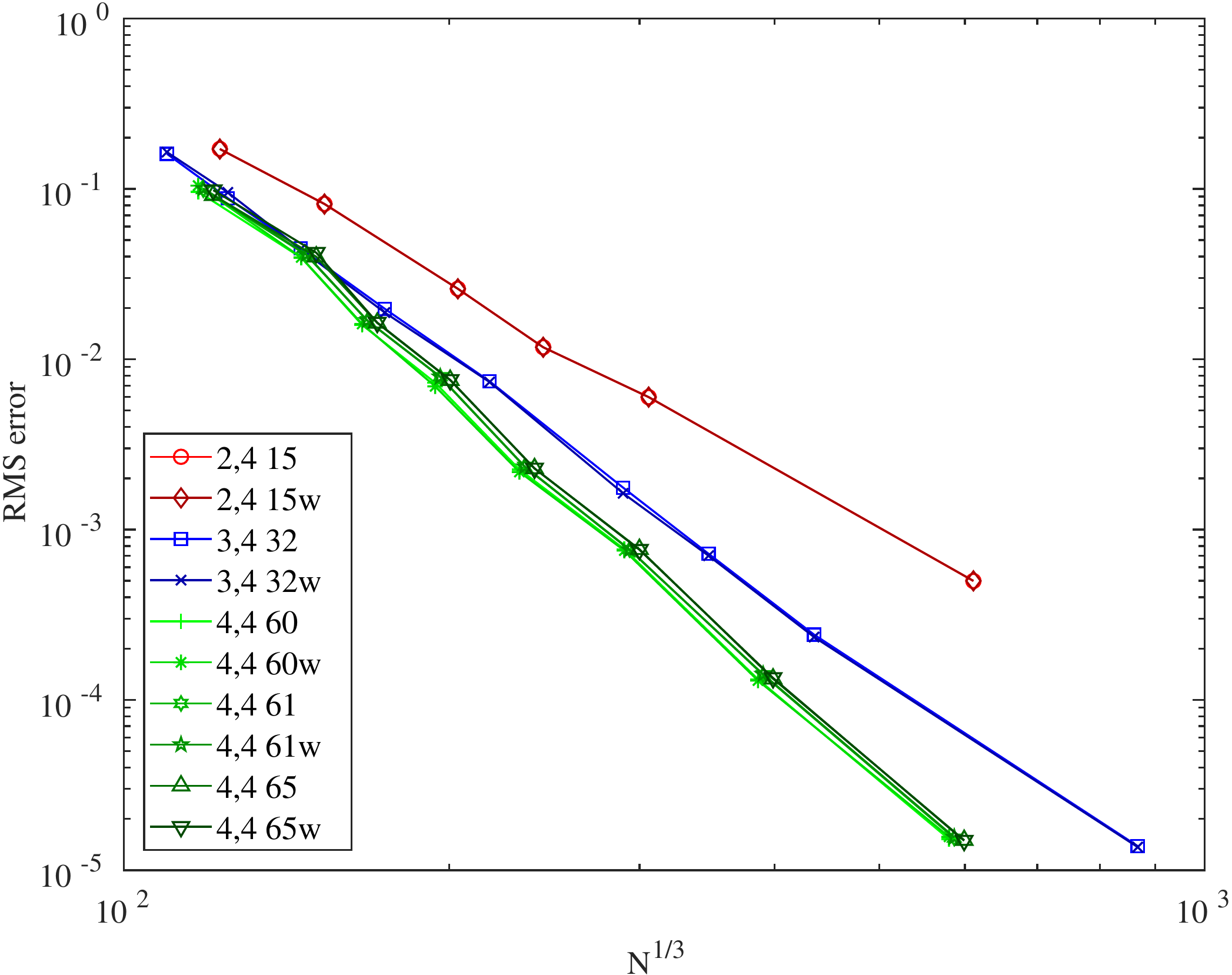}
\end{subfigure} \,\,
\begin{subfigure}[b]{0.45\textwidth}
  \includegraphics[width=\textwidth]{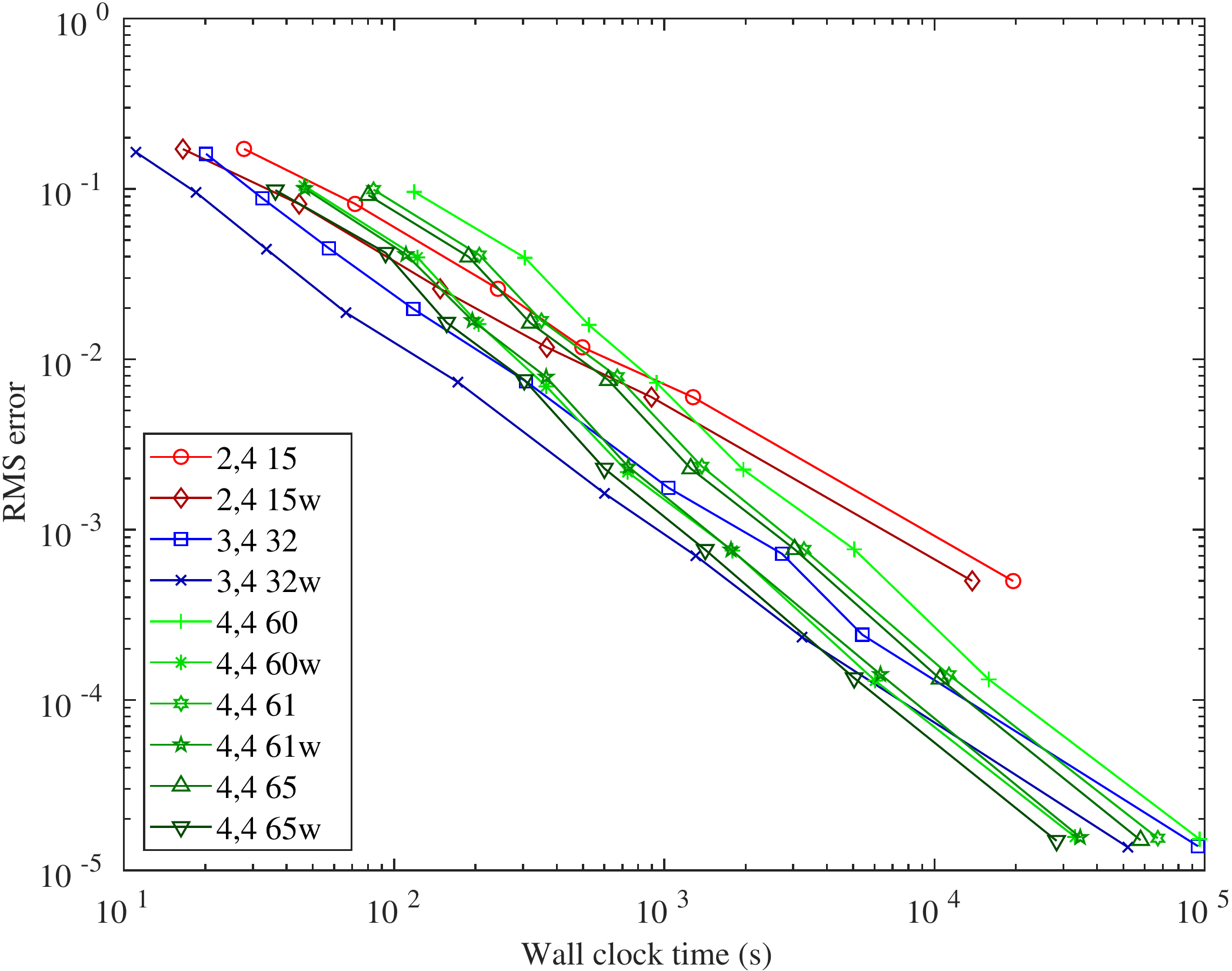}
\end{subfigure}
\caption{RMS errors for the elastic test case as a function of the cube root of the number of scalar degrees of freedom (left) and as a function of the wall clock time (right). In the legend, [$p,K\;n$] refers to the element of degree $p$ with $n$ nodes, combined with an order-$K$ time-stepping scheme. Suffix $w$ denotes elements for which the stiffness matrix was evaluated using the quadrature-based approach described in Section \ref{sec:alg}, while for the other elements we used the exact-integral algorithm.}
\label{fig:rmsEl}
\end{figure}

\begin{table}[h]
\caption{Results for the elastic test case, showing number of scalar degrees of freedom $N$, number of time steps $N_{\Delta t}$, wall clock time, and RMS error of the degree-$p$ $n$-node mass-lumped finite element method with stiffness matrix evaluation using exact integration ($n'$ = -), a new quadrature rule from this paper ($n'$ bold), or a degree-$(p+p'-2)$ accurate rule taken from \cite{zhang09}, with $n'$ the number of quadrature points.}
\label{tab:rmsEl}
\begin{center}
\begin{tabular}{c c c |r r r r}
$p$ 	& $n$ 	& $n'$ 		& $N$ 		& $N_{\Delta t}$	& time (s) 	& RMS error \\ \hline\hline
2 	& 15 		& -	 		& $24.6 \times10^{6}$	& 366 	& 988	& $8.24 \times10^{-3}$ \\
 	&	 	& 14	 		& 					& 371	& 750	& $8.24 \times10^{-3}$ \\ \hline
3 	& 32 		& -	 		& $10.4 \times10^{6}$ 	& 338	& 307	& $7.41 \times10^{-3}$ \\
 	&	 	& \textbf{21}	&  					& 327	& 172 	& $7.34 \times10^{-3}$ \\
 	&	 	& 24			&  					& 336	& 204	& $7.39 \times10^{-3}$ \\ \hline
4 	& 60 		& -	 		& $7.3 \times10^{6}$ 	& 744 	& 725	& $7.27 \times10^{-3}$ \\
 	&	 	& \textbf{51}	&  					& 697	& 366 	& $6.93 \times10^{-3}$ \\
 	&	 	& 59			&  					& 723	& 407	& $7.34 \times10^{-3}$ \\ \hline
4 	& 61 		& -	 		& $7.6 \times10^{6}$ 	& 631 	& 670	& $7.89 \times10^{-3}$ \\
 	&	 	& \textbf{60}	&  					& 614	& 365	& $7.87 \times10^{-3}$ \\
 	&	 	& 79			&  					& 614	& 471 	& $7.88 \times10^{-3}$ \\ \hline
4 	& 65 		& -	 		& $8.0 \times10^{6} $	& 568 	& 621	& $7.54 \times10^{-3}$ \\
 	&	 	& \textbf{60}	&  					& 553	& 303 	& $7.56 \times10^{-3}$ \\
 	&	 	& 79			&  					& 553	& 388 	& $7.99 \times10^{-3}$
\end{tabular}
\end{center}
\end{table}

We tested the methods on multiple unstructured meshes and simulated over the time interval $(-0.3,0.6)\;$s, using the time-stepping algorithm as in the previous test case and omitting the initial time steps where the magnitude of the wavelet is smaller than $10^{-16}$. Simulations were also carried out in the same environment as in the previous test case. The RMS error is based on the errors at all receivers and for all directional components and is plotted against the cube root of the number of scalar degrees of freedom $N$ and elapsed time in Figure \ref{fig:rmsEl}. Table \ref{tab:rmsEl} also lists the wall clock time and number of time steps for simulations with a RMS error of around $8\times 10^{-3}$ and includes results of simulations where a degree-$(p'+p-2)$ accurate quadrature rule taken from \cite{zhang09} is used, with $p'$ the highest polynomial degree of the enriched element space. The left graph of Figure \ref{fig:rmsEl} confirms that the methods converge with optimal order. It also shows that there is hardly any difference in accuracy between using the quadrature-rule approach or exact integration algorithm for evaluating the stiffness matrix. Table \ref{tab:rmsEl} and the right graph of Figure \ref{fig:rmsEl} show that for the degree-3 and degree-4 elements, the quadrature-based algorithm reduces the computational cost by more than a factor 1.5, while for the degree-2 element, this algorithm also results in a moderate speed up. Furthermore, Table \ref{tab:rmsEl} also illustrates that the new quadrature rules presented in this paper are more efficient than those currently available in the literature.

\section{Conclusion}
\label{sec:conclusion}
We presented new and efficient quadrature rules for evaluating the stiffness matrices of mass-lumped tetrahedral elements for wave propagation modelling. These quadrature rules can significantly reduce the number of computations compared to algorithms that evaluate the stiffness matrix using exact integration, and can handle spatial parameters that vary within the element without loss of the optimal convergence rate. Obtaining these quadrature rules is not trivial, since degree-$p$ mass-lumped tetrahedral element spaces contain, apart from polynomials up to degree $p$, numerous additional higher-degree bubble functions when $p\geq 2$. To obtain efficient quadrature rules, we therefore carefully analyzed the stability and accuracy requirements needed to maintain optimal convergence rates. The resulting conditions are presented in this paper, and we prove that, if these conditions are met, the resulting method can maintain an optimal order of convergence, even when the spatial parameters vary within the element. We found quadrature rules that satisfy these conditions for recently developed mass-lumped tetrahedral elements of degrees two to four. 

For the degree-2 element, the quadrature rule with the least number of points we could find was the degree-5 accurate 14-point quadrature rule of \cite{grundmann78}, but for the degree-3 and degree-4 elements, we found new quadrature rules that require significantly less integration points than those currently available. A dispersion analysis shows that by using these quadrature rules, the accuracy and largest allowed time step size remain nearly the same. Several numerical examples also illustrate the accuracy and efficiency of the quadrature-based approach and its superiority to evaluating the integrals for the stiffness matrix exactly. In particular, the quadrature-based approach results in a computational speed-up of around a factor 1.5 in case of elastic waves. Furthermore, in case of a heterogeneous domain with spatial parameters that vary within the element, the quadrature-based approach results in optimal convergence rates, while exact integration combined with a piecewise constant approximation of the spatial parameters results in a convergence rate of at most order two.

\bibliographystyle{abbrv}
\bibliography{MLFEM}

\end{document}